\documentclass{article}
\usepackage{graphicx} 
\usepackage{mathrsfs}
\usepackage{amsthm}
\usepackage{amsfonts,amssymb}
\usepackage{amscd}
\usepackage{amsmath}
\usepackage[all]{xy}
\usepackage{enumerate}
\usepackage{textcomp}
\usepackage{xcolor}
\usepackage{tikz}
\usetikzlibrary{arrows.meta, decorations.pathreplacing, calc}

\usepackage{subcaption} 

\usetikzlibrary{trees,quotes,positioning} 

\newtheorem{thm}{Theorem}[section]
\numberwithin{equation}{thm}

\newtheorem{definition}[thm]{Definition}
\newtheorem{cor}[thm]{Corollary}

\newtheorem{proposition}[thm]{Proposition}
\newtheorem{rk}[thm]{Remark}
\newtheorem{lemma}[thm]{Lemma}
\newtheorem{setup}[thm]{Setup}

\newtheorem{Notation}[thm]{Notation}

\usepackage{graphicx} 

\title{Positivity of tangent bundle of weak Del Pezzo surfaces of degree $4$}
\author{Qimin Zhang}

\begin{document}
\date{}
\maketitle
\noindent\textbf{Abstract.}
In this paper, we prove that for any weak Del Pezzo surface $S$ of degree at least $4$, the tangent bundle $T_S$ is almost nef. For the proof, we use total dual VMRTs induced by conic bundle structures.

\section{Introduction}

The positivity properties of the tangent bundle play an important role in the
birational and differential–geometric classification of projective varieties.
For example, a famous result of Mori \cite{zbMATH03657924} states that a smooth projective
variety with ample tangent bundle is isomorphic to the projective space
$\mathbb{P}^n$.

Motivated by this theorem, weaker positivity conditions on the tangent bundle
have been widely studied. In particular, for projective manifolds with nef
tangent bundle, Campana and Peternell obtained a classification of smooth
surfaces with nef tangent bundle and conjectured that any projective manifold
with nef tangent bundle should be homogeneous
\cite{zbMATH04204552, zbMATH00572322}.

More recently, attention has turned to even weaker positivity conditions,
such as pseudo-effectivity and almost nefness of the tangent bundle. 
Recall that for a
projective manifold $X$, a vector bundle $E$ is said to be
\emph{pseudo-effective} if the tautological line bundle
$\mathcal{O}_{\mathbb{P}(E)}(1)$ is pseudo-effective.
We also recall that:
\begin{definition}\label{def:almost nef}
Let $X$ be a projective manifold. A vector bundle $E$ is said to be
\emph{almost nef} if there exists a countable family of proper closed subsets
$\{Z_i\subsetneq X\}$ such that for every curve $C \subset X$ with
$C\not\subset \bigcup_i Z_i$, the restriction $E|_C$ is nef.
\end{definition}

These two notions of positivity are closely related. On a projective manifold, a line bundle is pseudo-effective if and only if it is almost nef. In contrast, for vector bundles, almost nefness implies pseudo-effectivity, whereas the converse generally fails \cite{zbMATH00572322, zbMATH01911897, zbMATH06154993}.

There are several classification results for projective manifolds with almost
nef tangent bundle, although the general picture is still incomplete. In the
surface case, Iwai obtained a classification of smooth projective surfaces
without $(-1)$-curves whose tangent bundle is almost nef
\cite[Proposition~4.11]{zbMATH07558532}. More precisely, if $S$ is such a
surface, then $T_S$ is almost nef if and only if $S$ is one of the following:
up to a finite étale cover, an abelian surface; a minimal ruled surface over either $\mathbb{P}^1$ or an elliptic curve; $\mathbb{P}^2$.

For surfaces containing $(-1)$-curves the situation is more subtle, even for weak Del Pezzo
surfaces the situation is not clear. For a (strong) Del Pezzo surface $S$ of degree $d$  \cite[Theorem~1.2]{zbMATH07730475} shows that
\[
T_S \text{ is pseudo-effective } \Longleftrightarrow d\geq 4.
\]
Moreover, combining the techniques of
\cite{zbMATH07730475, arXiv:2408.14411}, we have the following: for a Del Pezzo surface $S$ of degree $d$,
\[
T_S \text{ is almost nef } \Longleftrightarrow d\geq 4.
\]
The case of degree $4$ is thus a boundary case for this theory, it is also interesting for the existence for the Lagrangian fibration of $\Omega_X$ \cite{zbMATH08139336,zbMATH07969268}. 
In this paper, we prove the following result.

\begin{thm}[Main Theorem]\label{thm:Main}
Let $S$ be a weak Del Pezzo surface of degree~$4$.
Then the tangent bundle $T_S$ is almost nef.
\end{thm}
As an immediate corollary of Theorem \ref{thm:Main}, if $S$ is a weak Del Pezzo surface of degree at least $4$, then $T_S$ is almost nef.

For the strict Del Pezzo surfaces the result is based on a decomposition of elements in $|\mathcal{O}_{\mathbb{P}(T_S)}(2)|$. There are five pairs of total dual VMRTs, such that the sum of each pair is $|\mathcal{O}_{\mathbb{P}(T_S)}(2)|$. However, for weak Del Pezzo surfaces of degree~4, we do not have this good property. Thus, understanding the almost nefness of the tangent bundle therefore requires a more refined geometric analysis. 
\medskip

\noindent
\textbf{Strategy of the proof.}
In the strategy, we work on a weak Del Pezzo surface $S$ of degree $4$.
Our approach is inspired by \cite{zbMATH07730475, arXiv:2408.14411}. 
Let 
\[
\pi:\mathbb{P}(T_S)\rightarrow S
\]
be the natural projection. The basic idea is to construct an effective divisor
\[
D \in |k\zeta - \pi^*E|,
\]
where $k\in \mathbb{Z}_{>0}$ and $E$ is an effective divisor on $S$, such that
$\zeta|_D$ is pseudo-effective. We refer to
Proposition~\ref{all} for a precise formulation.

By \cite[Corollary~2.13]{zbMATH07730475}, one can construct effective divisors
\[
D_i \in |\zeta - \pi^*T_{f_i}|
\]
associated with conic bundle structures
$f_i:S\rightarrow \mathbb{P}^1$. These divisors can be interpreted as total
dual VMRTs. The strategy is therefore to find several such divisors $D_i$
satisfying:
\begin{itemize}
\item $\zeta|_{D_i}$ is pseudo-effective for each $i$;
\item the sum $\sum D_i$ belongs to a linear system of the form
$|k\zeta - \pi^*E|$ with $E$ effective.
\end{itemize}

To implement this strategy, we proceed as follows. In Section~2, we establish a
criterion to ensure that $\zeta|_D$ is pseudo-effective (Proposition \ref{pro:zeta on D} (3)). The main part of our effort is then to describe all the conic bundle structures on $S$:
let $f:S\rightarrow \mathbb{P}^1$ be a conic bundle structure, and let $F$ be a
general fibre. If $\mu:S\rightarrow \mathbb{P}^2$ is the blow-up map and $H$
denotes the hyperplane class on $\mathbb{P}^2$, then
\[
F\cdot \mu^*H \in \{1,2\}
\]
by Proposition~\ref{pro:conic bundles looks like}. Accordingly, we say that
the conic bundle structure is of degree $1$ or $2$.

Conic bundle structures of degree $1$ always exist on $S$. 
In contrast, the existence of conic bundle structures of degree $2$ is
equivalent to the existence of four points in good position, i.e.\ no three of
them are colinear in the generalized sense (see Definition~\ref{def:Colinear}).

We introduce the notation of good position to distinguish different types of weak Del Pezzo surfaces (Definition \ref{def:good position}), it might be interesting for weak Del Pezzo surfaces of other degree. 

In Section~5, we show that for every divisor $D_i$ arising from a conic bundle
structure, the restriction $\zeta|_{D_i}$ is pseudo-effective.

Finally, in Section~6, we prove the main theorem. If there exist four points in
good position, we construct divisors $D_1$ and $D_2$ corresponding to conic
bundle structures of degree $1$ and $2$, respectively, such that
\[
D_1 + D_2 \in |2\zeta - \pi^*(T_{f_1}+T_{f_2})|
\]
and $T_{f_1}+T_{f_2}$ is effective. If no such four points exist, then we show
that there exists a degree $1$ conic bundle structure $f$ such that
\[
D_1 \in |\zeta - \pi^*T_f|
\]
and $T_f$ is effective. In both cases, we conclude that $T_S$ is almost nef.

\section*{Acknowledgements}

The author would like to express his sincere gratitude to Prof.~Andreas Höring
and Prof.~Junyan Cao for their guidance, encouragement, and many helpful
discussions. The author is supported by a doctoral fellowship from the École
Doctorale SFA at Université Côte d’Azur.

\section{Preliminaries}
We work over the complex numbers.
All varieties are assumed to be projective and irreducible.
For general background we refer to \cite{Hartshorne1977}.
For divisors $E_1,E_2$ we write $E_1\ge E_2$ if
$E_1-E_2$ is pseudo-effective.
  \begin{definition}\label{def:psf}
Let $X$ be a projective manifold and $L$ a line bundle on $X$.
Let $D=\sum a_iD_i$ be an effective divisor with $D_i$ irreducible.
We say that $L|_D$ is pseudo-effective if
$L|_{D_i}$ is pseudo-effective for every component $D_i$.
\end{definition}
\begin{proposition}\label{all}
    Let $S$ be a smooth surface, and let $\pi:\mathbb{P}(T_S)\rightarrow S$ be the natural projection and  $\zeta$ be the tautological bundle on $\mathbb{P}(T_S)$. Assume there exists a positive integer $k$ and an effective divisor $E$ on $ S
    $ such that 
    \begin{enumerate}[(1)]
        \item $(k\zeta-\pi^*E)$ is effective
        \item There exists a divisor $D\in |k\zeta-\pi^*E|$, such that $\zeta|_D$ is pseudo-effective
    \end{enumerate}
    Then $T_S$ is almost nef.
    \begin{proof}

Assume that $T_S$ is not almost nef.
Then there exists a covering family of curves
$\{C_\alpha\}_{\alpha\in\Delta}$ such that
$T_S|_{C_\alpha}$ is not nef. By Definition \ref{def:almost nef}, the divisor 
\[\mathcal{O}_{\mathbb{P}(T_S|_{C_{\alpha}})}(1)\simeq \zeta |_{\pi^{-1}(C_{\alpha})}\]
is not nef. So for each $\alpha\in \Delta$, there exists a curve $\widetilde{C}_{\alpha}\subset \pi^{-1}(C_{\alpha})$, such that $\zeta \cdot \widetilde{C_{\alpha}} < 0$. 

Since $\{C_{\alpha}\}_{\alpha\in \Delta}$ is a covering family, it gives a movable class whose number has positive intersection number with $E$. Up to assuming that all the curves of the covering family are in the movable class, we have
for all $\alpha\in \Delta$,
\[
C_{\alpha}.E\geq 0
\] 
Thus,
\[
D \cdot \widetilde{C_{\alpha}} = (k\zeta - \pi^*E) \cdot \widetilde{C_{\alpha}} < 0,
\]
and therefore $\widetilde{C_{\alpha}} \subset D$, for all $\alpha\in \Delta$.

Denote the closure of $\bigcup_{\alpha\in I}\widetilde{C_{\alpha}}$ in the Zariski topology by $D'$. By \cite{zbMATH06154993}, a pseudo-effective divisor has non-negative intersection with all movable curves. So $\zeta|_D'$ is not pseudo-effective. However, by assumption, the restriction $\zeta|_D$ is pseudo-effective. This implies $\zeta|_{D'}$ is pseudo-effective by Definition \ref{def:psf}. We get a contradiction.
    \end{proof}
\end{proposition}

\begin{Notation}\label{notation:relative tangent bundle}
Let $f:S\to C$ be a fibration from a smooth surface
to a smooth curve.

\begin{enumerate}[(1)]
\item
Let
\[
\Delta=\{a\in C \mid f^{-1}(a)\ \text{is singular}\}.
\]

\item
Define ramification divisor of $f$
\[
R(f)=\sum_{a\in\Delta}
\left(f^{-1}(a)-(f^{-1}(a))_{\mathrm{red}}\right).
\]

\item
The relative tangent bundle $T_f:=T_{S/C}$ is defined as
the saturation of $\Omega_{S/C}^*$ in $T_S$.

\item \cite{arXiv:2408.14411}
There exists a zero-dimensional scheme $\Gamma$
such that
\begin{equation}\label{exact}
0\to T_{S/C}\to T_S\to
(-f^*K_C-R(f))\otimes\mathcal I_\Gamma\to0 .
\end{equation}
\end{enumerate}  
\end{Notation}
\begin{proposition}\label{Pro:VMRT}\cite{arXiv:2408.14411}
    Let $S$ be a smooth surface and $f: S \rightarrow \mathbb{P}^1$ be a fibration. Then there exists an effective divisor $D\subset \mathbb{P}(T_S)$, such that $D\in |\zeta-\pi^*T_{f}|$, where $T_{f}$ is the relative tangent bundle defined in Notation \ref{notation:relative tangent bundle}.

    \begin{proof}
    Take the determinant of the exact sequence \eqref{exact}, it follows that $T_{f}=-K_S+f^*K_{\mathbb{P}^1}+R(f)$. Twist the exact sequence \eqref{exact} by $T_{f}^*$, we get
    \begin{align}
        0\rightarrow \mathcal{O}_S \rightarrow T_S\otimes T_{f}^*  \rightarrow (-f^*K_{\mathbb{P}^1}-R(f))\otimes T_{f}^*\otimes \mathcal{I}_{\Gamma}  \rightarrow 0
    \end{align}
    Define $D:= \mathbb{P}((-f^*K_{\mathbb{P}^1}-R(f))\otimes \mathcal{I}_{\Gamma}) \subset \mathbb{P}(T_S)$. 
    Then $D\simeq \mathbb{P}((-f^*K_{\mathbb{P}^1}-R(f))\otimes T_{f}^*\otimes \mathcal{I}_{\Gamma})$, therefore $D\in |\mathcal{O}_{T_S\otimes T_{f}^*}(1)|=|\zeta-\pi^*T_{f}|$. 
    \end{proof}
\end{proposition}
    
\begin{proposition}\label{Q-q}
    With the notation of Notation \ref{notation:relative tangent bundle} and let $D$ to be the effective divisor defined in Proposition \ref{Pro:VMRT}. Assume that the fibres of $f$ have SNC supports. Then $\Gamma$ is a set of reduced points and there is birational morphism $q:D\rightarrow S$ induced by blowing up the points in $\Gamma$, such that $q=\pi|_D$.
    \begin{proof}
        Since the fibres of $f$ have SNC support, the scheme $\Gamma$ is reduced. Since $\Gamma$ is reduced, by \cite[Chapter~II, \S7]{Hartshorne1977}, we have $Sym^n\mathcal{I}_{\Gamma}\simeq \mathcal{I}_{\Gamma}^n$ for all $n\in \mathbb{Z}_{>0}$. So
    \begin{equation}
        \begin{split}
            D&=Proj(\oplus_n Sym^n ((-f^*K_{\mathbb{P}^1}-R(f))\otimes \mathcal{I}_{\Gamma}))\\
            &\simeq Proj (\oplus_n Sym^n\mathcal{I}_{\Gamma}) \\
            &\simeq Proj (\oplus_n \mathcal{I}_{\Gamma}^n)\\
            &\simeq Bl_{\mathcal{I}_{\Gamma}} S
        \end{split}
    \end{equation}
    So D is isomorphic to the blow up of S along $\Gamma$. 
    \end{proof}
\end{proposition}

\begin{Notation}\label{notation:F_alpha}
    Let $S$ be a smooth surface, and $f:S\rightarrow \mathbb{P}^1$ be a fibration. Assume that the fibres of $f$ have SNC support, then 
    \begin{enumerate}
        \item [(1)] Denote all the singular fibres by $\{F_{\alpha}\}_{\alpha\in \Delta}$.
        \item [(2)] Fix an $\alpha\in \Delta$, write $F_{\alpha}=\sum\limits_{i=1}^m a_iF_{i,\alpha}$, where $a_i>0$ and $F_{i,\alpha}$ are irreducible divisors. Then define $M(F_{\alpha}):=\max_{i=1}^m\{a_i\}$
        \item [(3)] Since each fibre has SNC support, by Proposition \ref{Pro:VMRT}, $\Gamma$ is a reduced 0-scheme. So we can write $\Gamma=\{q_1,...,q_s\}$.
        \item [(4)] Define $Q_i:=q^{-1}(q_i)$ for all $1\leq i\leq s$, where $q$ is defined in Proposition \ref{Q-q}. Since $q$ is blow up at smooth points, $Q_i$ are smooth curves. 
    \end{enumerate}
    
\end{Notation}

\begin{proposition}\label{pro:zeta on D}
    With the notation of Notation \ref{notation:F_alpha} and let $D$ to be the divisor defined in Proposition \ref{Pro:VMRT}. Then the following holds:
    \begin{enumerate}[(1)]
        \item 
        $\zeta|_{D} \simeq \sum\limits_{\alpha \in \Delta} \left( -\sum\limits_{q_i \in F_\alpha} Q_i - q^*(F_\alpha - (F_\alpha)_{\mathrm{red}}) \right) + 2q^* F$
        \item For each $\alpha\in \Delta$, the divisor
        \[
        -\sum_{q_i \in F_\alpha} Q_i - q^*(F_\alpha - (F_\alpha)_{\mathrm{red}}) + (1 - \frac{1}{2M(F_{\alpha})}) q^* F_\alpha
        \]
        is effective.
        \item If $\sum\limits_{\alpha\in \Delta}(1 - \frac{1}{2M(F_{\alpha})})\leq 2$, then $\zeta|_D$ is pseudo-effective.
    \end{enumerate}
    \begin{proof}
        \begin{enumerate}[(1)]
\item Since
\[
K_{\mathbb{P}(T_S)}
=\pi^*\bigl(K_S+\det(T_S)\bigr)-2\zeta
=-2\zeta,
\]
the adjunction formula on $\mathbb{P}(T_S)$ yields
\begin{align}\label{AD}
    K_{D} = (-2\zeta + D)|_{D}
\end{align}
Thus,
\begin{align*}
\zeta|_{D}
&= -K_D + (D - \zeta)|_D
&& \text{by }\eqref{AD} \\
&\simeq -K_D - \pi^* T_{f}\big|_{D}
&& \text{since } D \simeq \zeta - \pi^* T_{f} \\
&= -K_D - \bigl(
      -\pi^* K_S
      + \pi^* f^* K_{\mathbb{P}^1}
      + \pi^* R(f)
    \bigr)\big|_{D}
&& \text{by \eqref{exact}} \\
&\simeq -K_D + \pi^* K_S\big|_{D}
   + 2 \pi^* F\big|_{D}
   - \pi^* R(f)\big|_{D}
&&  K_{\mathbb{P}^1} = -2\mathcal{O}_{\mathbb{P}^1}(1) \\
&= \bigl(-K_D + q^* K_S - q^* R(f)\bigr) + 2 q^* F
&& \text{rewriting via } q \\
&= -\sum_{i=1}^s Q_i - q^* R(f) + 2 q^* F
&&  K_D = q^* K_S+\sum_{i=1}^s Q_i \\
&= \sum_{\alpha \in \Delta}
   \left(
     - \sum_{q_i \in F_\alpha} Q_i
     - q^*\bigl(F_\alpha - (F_\alpha)_{\mathrm{red}}\bigr)
   \right)
   + 2 q^* F
\end{align*}

\item Up to renumbering, we can assume 
$\Gamma_\alpha := \Gamma \cap F_\alpha = \{q_1, \dots, q_n\}$.
Since $F_\alpha$ has simple normal crossing support, each $q_j$ is an
intersection point of distinct irreducible components of $F_\alpha$,
for $1 \leq j \leq n$.
Hence we have
\begin{equation}\label{eq:pullback-E}
q^* \left( \sum_{i=1}^m F_{i,\alpha} \right)
= \sum_{i=1}^m \widetilde{F}_{i,\alpha} + \sum_{j=1}^n 2 Q_j ,
\end{equation}
where $\widetilde{F}_{i,\alpha} \subset D$ are the strict transforms of $F_{i,\alpha}$.
Then we compute as follows:
\begin{align*}
& -\sum_{q_i \in F_\alpha} Q_i
  - q^*\bigl(F_\alpha - (F_\alpha)_{\mathrm{red}}\bigr)
  + \left( 1 - \frac{1}{2M(F_\alpha)} \right) q^* F_\alpha
&&  \\
= {}&
 -\sum_{q_i \in F_\alpha} Q_i
 + q^*(F_\alpha)_{\mathrm{red}}
 - \frac{1}{2M(F_\alpha)} q^* F_\alpha
&& \\
= {}&
 -\sum_{q_i \in F_\alpha} Q_i
 + q^*\!\left( \sum_{i=1}^m F_{i,\alpha} \right)
 - \frac{1}{2M(F_\alpha)} q^*\!\left( \sum_{i=1}^m a_i F_{i,\alpha} \right)
&&   F_\alpha = \sum_{i=1}^m a_i F_{i,\alpha} \\
\geq {}&
 -\sum_{q_i \in F_\alpha} Q_i
 + q^*\!\left( \sum_{i=1}^m F_{i,\alpha} \right)
 - \frac{1}{2} q^*\!\left( \sum_{i=1}^m F_{i,\alpha} \right)
&& a_i \le M(F_\alpha) \\
= {}&
 -\sum_{q_i \in F_\alpha} Q_i
 + \frac{1}{2} q^*\!\left( \sum_{i=1}^m F_{i,\alpha} \right)
&& \\
= {}&
 \frac{1}{2} \sum_{i=1}^m \widetilde{F}_{i,\alpha} \geq 0
&& \text{by \eqref{eq:pullback-E}} .
\end{align*}

\item It follows immediately from (1) and (2):
\begin{align}
    \begin{split}
        \zeta|_{D} &\simeq \sum\limits_{\alpha \in \Delta} \left( -\sum\limits_{q_i \in F_\alpha} Q_i - q^*(F_\alpha - (F_\alpha)_{\mathrm{red}}) \right) + 2q^* F\\
        & \simeq \sum\limits_{\alpha \in \Delta} \left( -\sum\limits_{q_i \in F_\alpha} Q_i - q^*(F_\alpha - (F_\alpha)_{\mathrm{red}})+\left( 1 - \frac{1}{2M(F_\alpha)} \right) q^* F_\alpha \right)  \\
        &+\left(2-\sum_{\alpha\in \Delta}\left( 1 - \frac{1}{2M(F_\alpha)} \right) \right)q^* F \geq 0
\end{split}
\end{align}
        \end{enumerate}
    \end{proof}
    
\end{proposition}

\section{Simple birational morphisms}

\begin{definition}\label{def:simple morphism}
Let $S$ and $S'$ be smooth projective surfaces, and let
\[
\mu \colon S \to S'
\]
be a birational morphism. Denote by $\operatorname{Exc}(\mu)$ the exceptional locus of $\mu$.
We say that $\mu$ is \emph{simple} if
\[
\operatorname{Exc}(\mu) = \bigcup_{i \in B} E_i,
\]
where each $E_i$ is a smooth rational curve with self-intersection
\[
E_i^2 = -1 \text{ or } -2.
\]
\end{definition}

\begin{definition}
    Let $\mu:S\rightarrow S'$ be a simple birational morphism. Then define the exceptional graph of $\mu$ as follows: 
    \begin{enumerate}[(1)]
        \item Each vertex $q_i$ corresponds to an exceptional curve $E_i$.
        \item A vertex $q_i$ is represented as "$\circ$" if $E_i$ is a $(-1)$-curve and a vertex $q_i$ is represented as "$\bullet$" if $E_i$ is a $(-2)$-curve.
        \item An edge connects $q_i$ and $q_j$ if $E_i \cap E_j \neq \emptyset$.
    \end{enumerate}
\end{definition}

\begin{proposition}\label{graph}
    Let $\mu:S\rightarrow S'$ be a simple birational morphism, then 
    \begin{enumerate}[(1)]
\item
Each connected component of the exceptional graph of $\mu$ is of the following form:
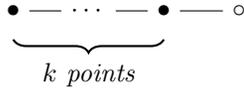
\[
\begin{tikzpicture}[baseline=(A.base)]
    \node (A) at (0,0) {$\bullet$};
    \node (B) at (1,0) {$\cdots$};
    \node (C) at (2,0) {$\bullet$};
    \node (E) at (3,0) {$\circ$};
    \draw (A) -- (B) -- (C) -- (E);

    \draw [decoration={brace, amplitude=5pt, mirror}, decorate, thick] 
      ([yshift=-5pt]A.south) -- ([yshift=-5pt]C.south)
      node [midway, below=5pt] {$k$ points};
\end{tikzpicture}
\]

\item Define $l$ to be the maximal length of the connected components and define $r$ to be number of the connected components. Then $\mu$ admits a factorization
\[
S = S_l \xrightarrow{\mu_l} S_{l-1} \xrightarrow{\mu_{l-1}} \cdots
\xrightarrow{\mu_2} S_1 \xrightarrow{\mu_1} S',
\]
together with a sequence of subsets
\[
A_l \subset A_{l-1} \subset \cdots \subset A_1 = \{1,\dots,r\},
\]
where each $\mu_j$ is the blow-up of the set of points
\[
\{\, p_{i,j} \mid i \in A_j \,\}, \qquad 1 \le j \le l.
\]

\end{enumerate}

\begin{proof}
(1) By Castelnuovo’s theorem, $S\rightarrow S'$ factors as a sequence of blow ups. In a sequence of blow-ups,
a new center cannot lie on a $(-2)$-curve, since blowing up such a point
would produce a curve of self-intersection $\leq -3$, contradicting the
simplicity assumption. Hence each blow-up either occurs on a $(-1)$-curve
or at a point disjoint from the exceptional locus. 

(2) The factorization follows from (1) by contracting the $i$-th point from right to left on
each chain, which yields the sets $A_j$.
\end{proof}
    
\end{proposition}

\begin{Notation}\label{notation:C_i,j}
Let $\mu:S\to S'$ be a simple birational morphism. By
Proposition \ref{graph}, the exceptional locus of $\mu$ is a disjoint
union of $r$ chains. We fix the following notation.

\begin{enumerate}[(1)]
\item For each $1\le i\le r$, let
\[
p_{i,1},\dots,p_{i,l(i)}
\]
be the sequence of infinitely near points corresponding to the $i$-th
chain, where $l(i)$ is its length. Set
\[
l:=\max_{1\le i\le r} l(i).
\]

\item For $1\le i\le r$ and $1\le j\le l(i)$, let
\[
C_{i,j}\subset S_j
\]
be the exceptional curve of the blow-up at $p_{i,j}$.

\item Let $E_{i,j}\subset S$ be the strict transform of $C_{i,j}$ on $S$.

\item We define
\[
\mu^*C_{i,j}:=(\mu_{j+1}\circ\cdots\circ\mu_l)^*C_{i,j}.
\]
\end{enumerate}

\medskip

\noindent
With this notation, the $i$-th chain has the following form:
\begin{center}
\begin{tikzpicture}[
  baseline=(A.base),
  node distance=12mm,
  every node/.style={font=\normalsize},
  ex/.style={circle, fill=black, inner sep=1.4pt},
  nonex/.style={circle, draw=black, fill=white, inner sep=1.4pt},
  lab/.style={font=\small, inner sep=1pt},
  conn/.style={thick, line cap=round, line join=round}
]

\node[ex]    (A) at (0,0) {};
\node        (B) at (1,0) {$\cdots$};
\node[ex]    (C) at (2,0) {};
\node[nonex] (E) at (3,0) {};

\draw[conn] (A) -- (B) -- (C) -- (E);

\node[lab, below=6pt of A] {$E_{i,1}$};
\node[lab, below=6pt of B] {$\cdots$};
\node[lab, below=6pt of C] {$E_{i,l(i)-1}$};
\node[lab, below=6pt of E] {$E_{i,l(i)}$};

\end{tikzpicture}
\end{center}
\end{Notation}
We summarize the notations in a commutative diagram:
\begin{center}
\vspace{0.5em}
$\xymatrix{
    S_{k+1} \ar[r]^{\mu_{k+1}} & S_{k} \\
    C_{i,k+1} \ar@{^{(}->}[u] \ar[dr]_{\mu_{k+1}} & C_{i,k} \ar@{^{(}->}[u] \ar[d]^{\rotatebox{90}{$\in$}} \\
    & p_{i,k+1}
}$
\end{center}

\begin{proposition}\label{Pro:!}
We adopt the Notation~\ref{notation:C_i,j}. The following holds: 
\begin{enumerate}[(1)]
    \item $\mu^*C_{i,j}=E_{i,j}+...+E_{i,l(i)}$, for all $1 \leq i \leq r$, $1\leq j\leq l(i)$
    \item $(\mu^*C_{i,j})^2 = -1$ for all $1 \leq i \leq r$, $1\leq j\leq l(i)$
    \item $\mu^*C_{i_1,j_1} \cdot \mu^*C_{i_2,j_2} = 0$ for all $(i_1,j_1) \neq (i_2,j_2)$
    \item $\mu^*L \cdot \mu^*C_{i,j} = 0$ for all $1 \leq i \leq r$, $1\leq j\leq l(i)$, and $L\in \operatorname{Pic}(S')$
    \item $-K_S = -\mu^*K_{S'} - \sum\limits_{i=1}^r\sum\limits_{j=1}^{l(i)} \mu^*C_{i,j}$
    \item $\{\mu^*\operatorname{NS}(S')\}\cup \{\mu^*C_{i,j}|1 \leq i \leq r$, $1\leq j\leq l(i)\}$ form a $\mathbb{Z}$-basis of $\operatorname{NS}(S)$
\end{enumerate}
    
    \begin{proof}
\textbf{(1)}
    Without loss of generality, we assume $i=1$. For
$1\le j\le l(1)-1$, let $C_{1,j}^m$ be the strict transform of $C_{1,j}$ on
$S_m$ for $m>j$. Since $\mu_{j+1}$ blows up a point on $C_{1,j}$, we have
\begin{align}\label{special structure-1}
    \mu_{j+1}^*C_{1,j}= C_{1,j+1}+ C_{1,j}^{j+1}
\end{align}
on $S_{j+1}$ 
By the construction in Proposition \ref{graph}, for every $k\ge j+1$ the point
$p_{1,k+1}$ does not lie on $C_{1,j}^k$, hence
\[
\mu_{k+1}^*C_{1,j}^k=C_{1,j}^{k+1}.
\]
for all $k\geq j+1$. Therefore the pull-back of $C_{1,j}^{j+1}$ to $S$ is just its strict transform,
namely $E_{1,j}$.
Pulling back formula \eqref{special structure-1} by $(\mu_{j+2}\circ...\circ \mu_{l})$, we get
\begin{align}\label{special structure-2}
    \mu^*C_{1,j}=\mu^*C_{1,j+1}+E_{1,j}
\end{align}
Formula \eqref{special structure-2} holds for all $1\leq j\leq l$, and obviously $\mu^*C_{1,l}=E_{1,l}$. We conclude by induction on $j$.

        \textbf{(2)} By the projection formula for $(\mu_{j+1}\circ...\circ\mu_{l})$, we obtain
        \[
        (\mu^*C_{i,j})^2 = (\mu_{j+1}\circ...\circ\mu_{l})^*C_{i,j} \cdot (\mu_{j+1}\circ...\circ\mu_{l})^*C_{i,j} = C_{i,j} \cdot C_{i,j} = (-1).
        \] 
        
        \textbf{(3)} If $i_1\neq i_2$, this is true because $\mu^*C_{i_1,j_1}$ and $\mu^*C_{i_2,j_2}$ have disjoint support. If $i_1=i_2$, without loss of generality, set $j_1>j_2$. Apply the projection formula for $(\mu_{j_2+1}\circ...\circ \mu_l)$:
        \begin{align}
            \begin{split}
                &\mu^*C_{i_1,j_1} \cdot \mu^*C_{i_2,j_2} \\
                =& (\mu_{j_2+1}\circ...\circ \mu_l)^* C_{i_1,j_1} \cdot (\mu_{j_2+1}\circ...\circ \mu_l)^*C_{i_2,j_2}\\
                =& C_{j_1} \cdot (\mu_{j_2+1} \circ \cdots \circ \mu_{j_1})^* C_{i_1,j_2} = 0
            \end{split}
        \end{align}
        
        The last equality holds because $C_{i_1,j_1}$ is the exceptional divisor of $\mu_{j_1}$ and $(\mu_{j_2+1} \circ \cdots \circ \mu_{j_1})^*C_{i_1,j_2}$ is the pull back of a divisor by $\mu_{j_1}$.
        
        \textbf{(4)} By the projection formula for $(\mu_{j+1}\circ...\circ\mu_l)$, we have:
        \[\mu^*L\cdot \mu^*C_{i,j}=\mu^*L\cdot (\mu_{j+1}\circ...\circ\mu_l)^*C_{i,j}=(\mu_1\circ ... \circ \mu_{j})^*L \cdot C_{i,j}\].
        The last term equals zero because $C_{i,j}$ is the exceptional divisor of $\mu_j$ and $(\mu_{j+1}\circ...\circ\mu_l)^*L$ is the pull back of a divisor by $\mu_j$.
        
        \textbf{(5)} 
        For convenience, we denote $C_{i,j}=0$, for all $l(i)+1\leq j\leq l$. 
        Note that 
        for $j = k+1$, one has
        $ -K_{S_{k+1}}= \mu_{k+1}^*(-K_{S_{k}}) - \sum\limits_{i=1}^r C_{i,k+1} $.
        Together with the formula $-K_{S_{1}} = \mu_{1}^*(-K_{S'}) - \sum\limits_{i=1}^r C_{i,1}$,
        This implies (5) by induction.
        
        \textbf{(6)} Since $\operatorname{NS}(S)$ is generated by $\mu^*\operatorname{NS}(S')$ and $\{E_{i,j}|1\leq i\leq r, 1\leq j\leq l(i)\}$, it has dimension $k+\sum\limits_{s=1}^rl(i)$, where $k=\operatorname{dim}\operatorname{NS}(S')$. Take a basis $L_1,...,L_k$ of $\operatorname{NS}(S')$. Suppose:
        \[
        \sum_{p=1}^kb_p \mu^*L_p + \sum_{s=1}^r\sum_{t=1}^{l(s)} a_{s,t} \mu^*C_{s,t} \simeq 0.
        \]
        Intersecting with $\mu^*C_{i,j}$ and using $(2),(3)$, we obtain:
        \[
        a_{i,j} (\mu^*C_{i,j})^2 + \sum_{(s,t) \neq (i,j)} a_{s,t} (\mu^*C_{s,t} \cdot \mu^*C_{i,j}) = -a_{i,j} = 0 \implies a_{i,j} = 0.
        \]
        This implies $\sum_{p=1}^kb_p L_p=\mu(\sum_{p=1}^kb_p \mu^*L_p)\simeq 0$. Since $L_1,...,L_k$ forms a basis of $\operatorname{NS}(S')$, we have $b_p=0$ for all $1\leq p\leq k$.
        Thus the divisors are linearly independent and form a basis of $\operatorname{NS}(S)$.
    \end{proof}
\end{proposition}
Blowups along infinitely near points will play a central role in this paper. This leads to the following definition.
\begin{definition}\label{general-belong}
With the notation as in Notation~\ref{notation:C_i,j}, let $D\subset S'$ be a smooth curve.
We say that $D$ passes through $p_{i,j}\in S_{j-1}$ in the generalized sense if
\[
\mu^*D - \sum_{t=1}^j \mu^*C_{i,t}
\]
is effective. In this case, we write $p_{i,j}\in_{gen} D$. Otherwise, we write $p_{i,j}\not\in_{gen}D$.
\end{definition}
From this definition, it follows that if $p_{i,j}\in_{gen} D$, then $p_{i,t}\in_{gen} D$ for all $1\leq t\leq j$.
\begin{proposition}\label{2nd def of belonging}
    Using the notation of Definition \ref{general-belong}, denote the strict transform of $D$ in $S$ by $\hat{D}$ and the strict transform of $D$ in $S_{j}$ by $\hat{D}_{j}$. Then\\
    (1) $\mu^*D=\hat{D}+\sum\limits_{p_{i,j} \in_{gen} D }\mu^*C_{i,j}$, and\\
    (2) $p_{i,j}\in_{gen}D$ if and only if $p_{i,j} \in \hat{D}_{j-1}$ on $S_{j-1}$.
    
\begin{proof}
 If the statement is true on each open set $U_i=S-\bigcup\limits_{s\neq i}\mu^{-1}(p_{s,1})$ with $1\leq i\leq r$, then it is true on $S$. So it suffices to treat the case $r=1$.
 
For all $1\leq t\leq j$, we have $\mu_{t}\circ...\circ \mu_{j-1}(p_{1,j})=p_{1,t}$ and $\mu_{t}\circ...\circ \mu_{j-1}(\hat{D}_{1,j})=\hat{D}_{t}$. So if $p_{1,j}\in \hat{D}_{1,j}$ on $S_{j-1}$, then $p_{1,t}\in \hat{D}_{t}$ with $1\leq t\leq j$. Take $k$ to be the largest integer such that $p_{1,j}\in\ \hat{D}_{j}$. 

Since $D$ is smooth, we have:
\begin{equation}\label{def:2nd def of belonging-1}
\mu_{t}^*\hat{D}_{t} = 
\begin{cases}
\hat{D}_{t+1} + C_{1,t} & \text{if } 1\leq t\leq k \\
\hat{D}_{t+1} & \text{if } k+1\leq t \leq l(1)
\end{cases}
\end{equation}
Pulling back the equation (\ref{def:2nd def of belonging-1}) by $\mu_{t+1}\circ...\circ \mu_{l}$ for $1\leq t\leq k$, we have 
\begin{align}\label{2nd def of belonging-2}
    (\mu_{t}\circ \mu_{t+1}\circ...\circ \mu_{l})^*\hat{D}_{t}=(\mu_{t+1}\circ...\circ \mu_{l})^*\hat{D}_{t+1}+\mu^*C_{1,t}
\end{align}
By definition, we have $D=\hat{D}_{1}$. Sum formula (\ref{2nd def of belonging-2}) up for $1\leq t\leq k$ to get:
\[\mu^*D=\mu^*\hat{D}_{1}=(\mu_{1}\circ...\circ \mu_{l})^*\hat{D}_{1}=(\mu_{k+1}\circ...\circ \mu_{l})^*\hat{D}_{k+1}+\mu^*C_{1,1}+\cdots+\mu^*C_{1,k}\]
By the maximality of $k$, for each $t\geq k+1$, we have $p_{1,t}\not\in \hat{D}_{1,t}$. This shows that $\mu_{t}^*\hat{D}_{t}=\hat{D}_{t+1}$.
So $(\mu_{k+1}\circ...\circ \mu_{l})^*\hat{D}_{k+1}=\hat{D}_l=\hat{D}$ and 
\begin{align}\label{eq:2nd def of belonging-3}
    \mu^*D=(\mu_{1}\circ...\circ \mu_{l})^*\hat{D}_{1}=\hat{D}+\mu^*C_{1,1}+\cdots+\mu^*C_{1,k}=\hat{D}+\sum_{ \{t| p_{1,t}\in \hat{D}_{t}\}} \mu^*C_{1,t}  
\end{align}

If $p_{1,j}\in_{gen}D$, then $p_{1,t}\in_{gen}D$ for $1\leq t\leq j$. Take $k'$ to be the largest integer such that $p_{i,j}\in_{gen}D$. So to prove (1) and (2), we only have to show $k=k'$. Yet it is immediate by Definition \ref{general-belong} and \eqref{eq:2nd def of belonging-3}.

\end{proof}
\end{proposition}

\begin{lemma}\label{positive}
    Let $s_i$ are integers such that $1\leq s_i \leq l(i)$ and let $j_{i,t} $ are integers such that $1\leq j_{i,1}< ...<j_{i,s_i}\leq l(i)$.  For a smooth curve $D\subset S'$, the following are equivalent:
    \begin{enumerate}[(1)]
        \item $\mu^*D - \sum_{i=1}^r\sum_{t=1}^{s_i} \mu^*C_{i,t}$ is effective;
        \item $\mu^*D - \sum_{i=1}^r\sum_{t=1}^{s_i} \mu^*C_{i,j_{i,t}}$ is effective;
        \item $\mu^*D - \sum_{i=1}^r s_i \mu^*C_{i,\ell(i)}$ is effective.
    \end{enumerate}
    \begin{proof}
    We only need to prove this for $r=1$. Denote $s_1=s$ and $j_{1,t}=j_t$.
        By Proposition~\ref{Pro:!}(1), we have
\[
\mu^*C_{1,1}\ge \mu^*C_{1,2}\ge \cdots \ge \mu^*C_{1,l(1)}.
\]
Hence
\[
\sum_{t=1}^s\mu^*C_{1,t}\ge \sum_{t=1}^s\mu^*C_{1,j_t}\ge s\,\mu^*C_{1,l(1)},
\]
which immediately gives $(1)\Rightarrow(2)\Rightarrow(3).$

By Proposition \ref{2nd def of belonging}, we have $\mu^*D=\hat{D}+\sum_{p_{1,j}\in_{gen}D}\mu^*C_{1,j}$, 
        where $\hat{D}$ is the strict transform of $D$ by $\mu$. Let $k$ be the largest index such that $p_{1,j}\in_{gen}D$. Then by Proposition \ref{Pro:!}(1), we have
        \[\mu^*D=\hat{D}+\sum_{1\leq j\leq k}\mu^*C_{1,j}=\hat{D}+\sum_{j=1}^{k}jE_{1,j}+\sum_{j=k+1}^{l(1)}kE_{1,j}\]
        By assumption $\mu^*D-s\mu^*C_{1,l(1)}$ is effective, so the divisor
        \[\hat{D}+\sum_{j=1}^{k}jE_{1,j}+\sum_{j=k+1}^{l(1)-1}kE_{1,j}+(k-s)E_{1,l(i)}\]
        is effective. So $k\geq s$. This shows that $p_{1,s}\in_{gen}D$. This proves $(3)\Rightarrow(1)$.
    \end{proof}
\end{lemma}

\begin{definition}\label{minimal}
   Let $A := \{p_{i,j} \mid 1 \leq i \leq r,\ 1 \leq j \leq l(i) \}$, and let $B$ be a subset of $A$. For every $1\leq k\leq r$, set $B_k:=\{p_{k,j}\in B|1\leq j\leq l(k)\}$ and $b_k:=\#\{B_k\}$. Define $\operatorname{Min}(B)$, the minimization of $B$, as follows:
   \[
   \operatorname{Min}(B):=\bigcup_{1\leq k\leq r} \{p_{k,1},...,p_{k,b_k}\}
   \]
   Then we say $B$ is minimal if $B=\operatorname{Min}(B)$.
\end{definition}

\begin{lemma}\label{lemma:passes at least}
Let $B:=\{p_{i_t,j_t}\mid 1\le t\le s\}$ be a subset of
$A := \{p_{i,j} \mid 1 \le i \le r,\ 1 \le j \le \ell(i)\}$.
\begin{enumerate}[(1)]
\item If $\mu^*D-\sum_{t=1}^s \mu^*C_{i_t,j_t}$ is effective, then $D$ passes through
the points in $\operatorname{Min}(B)$ in the generalized sense. 
\item[(2)] If there exists a curve $D\subset S'$ such that $D$ passes through the points in $B$ in the generalized sense and $D$ does not pass through the points in $A-B$, then $B$ is minimal.
\end{enumerate}
\end{lemma}

\begin{proof}
This is a local problem, so we only need to prove it on each open set $U_i=S-\bigcup\limits_{s\neq i}\mu^{-1}(p_{s,1})$ with $1\leq i\leq r$. Therefore without loss of generality assume $r=1$. 

(1) By Lemma \ref{positive}, we know that
\begin{align*}
    (\mu^*D-\sum_{t=1}^s \mu^*C_{1,j_t}) \text{ is effective} \Rightarrow (\mu^*D-\sum_{t=1}^s \mu^*C_{1,t}) \text{ is effective}
\end{align*}
So by Definition \ref{general-belong}, the curve $D$ passes through the points in $\operatorname{Min}(B)$ in the generalized sense.

(2) Denote $b:=\#B$. Let $k$ be the maximal index such that $p_{1,j}\in_{gen}D$. Then $D$ passes through at least $k$ points in the generalized sense. On the other hand, the curve $D$ passes through the points in $B$ exactly. So we have $b=k$. By Definition \ref{minimal}, this means $B$ is minimal.
\end{proof}

\begin{thm}\label{thm:decomposition-mu}
    Take $A$ as in Definition \ref{minimal}. Let $B\subset A$ be a minimal subset. Then we can factor $\mu$ as $\gamma_1\circ \gamma_2$, such that $\gamma_1$ is the blow up of the points in $B$ and $\gamma_2$ is the blow up of points in $A-B$.
    \begin{proof}
        It suffices to treat the case for $r=1$. Then $\mu$ is the successive blow up of points $p_{1,1},...,p_{1,l}$. Since $B$ is minimal, $B=\{p_{1,1}...,p_{1,b}\}$. So $\gamma_1$ is the blow up of $p_{1,1},...,p_{1,b}$ and $\gamma_2$ is the blow up of $p_{1,b+1},...,p_{1,l}$. 
    \end{proof}
\end{thm}

\begin{definition}\label{def:simple surface}
    We say that a smooth projective surface is simple if there exists a simple birational morphism $\mu:S\rightarrow \mathbb{P}^2$.
\end{definition}

\begin{thm}
    Let $S$ be a smooth projective rational surface with $-K_S$ nef. Then $S$ is a simple surface if $S$ is not isomorphic to $\mathbb{P}^1\times \mathbb{P}^1$ or $\mathbb{F}_2$.
    \begin{proof}
        It follows from the classical classification result of smooth projective surfaces with anti-canonical bundle nef, see more details in \cite{zbMATH02148824}.
     \end{proof}
\end{thm}
In particular, all the weak Del Pezzo surfaces except $\mathbb{F}_2$ and $\mathbb{P}^1\times \mathbb{P}^1$ are simple surfaces. 
\begin{definition}\label{def:Colinear}
    On a simple surface $S$, we say three points $p_{i_1,j_1},p_{i_2,j_2},p_{i_3,j_3}$ are colinear in the generalized sense if the divisor class $\mu^*H-\mu^*C_{i_1,j_1}-\mu^*C_{i_2,j_2}-\mu^*C_{i_3,j_3}$ is effective.
\end{definition}
\begin{rk}
    If there exists a line which passes through $p_{i_1,j_1},p_{i_2,j_2},p_{i_3,j_3}$ in the generalized sense, then these three points are colinear in the generalized sense and minimal.
    
    Vice versa, if $p_{i_1,j_1},p_{i_2,j_2},p_{i_3,j_3}$ are colinear in the generalized sense, then there exists a line which passes through $\operatorname{Min}\{p_{i_1,j_1},p_{i_2,j_2},p_{i_3,j_3}\}$ in the generalized sense.
\end{rk}
\section{Conic bundle structures on simple surfaces}

\begin{definition}\label{setup:Conic bundle structure}
    Let $S$ be a simple surface and let $f:S\rightarrow \mathbb{P}^1$ be a fibration. 
    Denote a general fibre of $f$ by $F$. Then we say $f$ is a conic bundle structure if 
    \[-K_S.F=2\]
    Moreover, since $S$ is a simple surface, there exists a birational morphism $\mu:S\rightarrow \mathbb{P}^2$. We say $f$ is of degree $d$ if 
    \[F.\mu^*H=d\]
    where $H$ is the hyperplane class in $\mathbb{P}^2$.
\end{definition}

\begin{lemma}\label{effective}
In the situation of Setup~\ref{setup:Conic bundle structure},  
the divisor class
\[
k\mu^*H-\sum_{t=1}^{2k}\mu^*C_{i_t,j_t}
\]
is effective for all \(k\ge 1\).
\end{lemma}

\begin{proof}
It suffices to prove the statement for \(k=1\). Let \((i_1,j_1)\) and \((i_2,j_2)\) be two indices.
Let $\{p_{i'_1,j'_1},p_{i'_2,j'_2}\}:=\operatorname{Min}\{p_{i_1,j_1},p_{i_2,j_2}\}$. Then either
\[
\{(i'_1,j'_1),(i'_2,j'_2)\}=\{(i'_1,1),(i'_1,2)\} \text{ or } \{(i'_1,1),(i'_2,1)\}
\]
In both cases, there exists a line \(l\subset\mathbb P^2\) passing through $p_{i'_1,j'_1}$ and $p_{i'_2,j'_2}$ in the generalized sense. Thus, by Lemma \ref{lemma:passes at least} and Proposition \ref{Pro:!} (1), we have 
\[
\mu^*H-\mu^*C_{i_1,j_1}-\mu^*C_{i_2,j_2}\geq \mu^*l-\mu^*C_{i'_1,j'_1}-\mu^*C_{i'_2,j'_2}\geq 0 
\]
\end{proof}

\begin{proposition}\label{pro:conic bundles looks like}
    Let $S$ be a simple surface and $f:S\rightarrow \mathbb{P}^1$ be a conic bundle structure of degree $d$. Then $f$ is induced by a linear system 
    \[
    |d\mu^*H-\sum_{t=1}^{d^2}\mu^*C_{i_t,j_t}|
    \]
    \begin{proof}
        Fix a general fibre $F$ of $f$. By Proposition \ref{Pro:!} (5), we can write $F_0$ as 
        \begin{align}\label{eq:conic bundles looks like-1}
            F\simeq a\mu^*H+\sum_{1\leq i\leq r}\sum_{1\leq j\leq l(i)}a_{i,j}\mu^*C_{i,j}
        \end{align}
        where $a,a_{i,j}$ are integers, for all $1\leq i\leq r$, $1\leq j\leq l(i)$. Since $f$ is degree $d$, by Proposition \ref{Pro:!} (2)(3), we have $a=F_0.\mu^*H=d$.
        
        Define $C:=\mu(F)\subset \mathbb{P}^2$ and $\hat{C}$ to be the strict transform of $C$ by $\mu$. Then $F=\hat{C}$, since $F\subset \mu^{-1}C$ and $F_0$ is irreducible. By Proposition \ref{2nd def of belonging} (1), we have 
        \begin{align}\label{eq:conic bundles looks like-2}
            F_0=\hat{C}=\mu^*C-\sum_{p_{i,j}\in_{gen} C }\mu^*C_{i,j} 
            \simeq d\mu^*H-\sum_{p_{i,j}\in_{gen} C }\mu^*C_{i,j}
        \end{align}
        Comparing the formulas \eqref{eq:conic bundles looks like-1} and \eqref{eq:conic bundles looks like-2}, we have $a_{i,j}\in \{-1,0\}$, for all $1\leq i\leq r$, $1\leq j\leq l(i)$.

        Recall that $F$ is a fibre, thus $F^2=0$. By Proposition \ref{Pro:!} (2)(3), we have 
        \[
        0=(F)^2=(a\mu^*H+\sum_{1\leq i\leq r}\sum_{1\leq j\leq l(i)}a_{i,j}\mu^*C_{i,j})^2=d^2-\sum_{i,j}a_{i,j}^2
        \]
        Thus, there are exactly $d^2$ elements of $\{a_{i,j}\}$ equal to $-1$ and the others are zero. 
    \end{proof}
\end{proposition}

\subsection{Conic bundle structures of degree 1}

\begin{proposition}\label{conic 1}
    Let $S$ be a simple surface. Using the notation from Section 3, for every $i\in \{1,...,r\}$, the linear system $|\mu^*H-\mu^*C_{i,1}|$ defines a conic bundle structure of degree 1:
    \[f_{i}:S\rightarrow \mathbb{P}^1\]
    The general fibre is the strict transform of a line passing through $p_{i,1}$. 
    \begin{proof}
    The birational morphism $\mu$ can be factored in a commutative diagram:
    \begin{center}
    \vspace{0.5em}
    $\xymatrix{
    S \ar[d]^{\tau} \ar[dr]^{\mu} & \\
    \mathbb{F}_1 \ar[r]_{\gamma} & \mathbb{P}^2 
}$
\end{center}
where $\gamma$ represents the blow up of $p_{i,1}$. 
     It is well known that $\mathbb{F}_1$ has a conic bundle structure $f'_{i}:\mathbb{F}_1\rightarrow \mathbb{P}^1$ and a general fibre is induced by a line on $\mathbb{P}^2$ which passes through $p_{i,1}$. The composition $f_{i}:=f'_{i}\circ \tau:S\rightarrow \mathbb{P}^1$ is a conic bundle structure and the general fibre is the strict transform of a line passing through $p_{i,1}$.

    \end{proof}
\end{proposition}

\begin{rk}
By Proposition \ref{pro:conic bundles looks like} and Proposition \ref{conic 1}, let $r$ be the number defined in Proposition \ref{graph}. Then there are exactly $r$ conic bundle structures of degree 1.
\end{rk}

\subsection{Good position and conic bundle structures of degree~2}
Configurations of weak Del Pezzo surfaces of degree 4 are classified, see \cite[Figure 2]{arXiv:2408.14411}. For overview, here we introduce a notation that is convenient for our construction:
\begin{definition}\label{def:good position}
    On a simple surface $S$, we say four points $p_{i_1,j_1},p_{i_2,j_2},p_{i_3,j_3},p_{i_4,j_4}$ are in good position if none of three of them are colinear in the generalized sense. (For the notation colinear in the generalized sense, see Definition \ref{def:Colinear}.)
\end{definition}

\begin{lemma}\label{lemma:GP:B,MB}
    Let $B=\{p_{i_1,j_1}, p_{i_2,j_2}, p_{i_3,j_3}, p_{i_4,j_4}\}$ be a subset of $A$, then the following two statements are equivalent: 
    \begin{enumerate}[(1)]
        \item The points in $B$ are in good position 
        \item The points in $\operatorname{Min}(B)$ are in good position 
    \end{enumerate}
    \begin{proof}
        By Lemma \ref{positive}, for any $1\leq j_1'\leq l(i_1)$ such that $(i_1,j_1')\neq (i_2,j_2),(i_3,j_3)$, the effectiveness of 
        \[
        \mu^*H-\mu^*C_{i_1,j_1}-\mu^*C_{i_2,j_2}-\mu^*C_{i_3,j_3}
        \]
        is equivalent to the effectiveness of 
        \[
        \mu^*H-\mu^*C_{i_1,j_1'}-\mu^*C_{i_2,j_2}-\mu^*C_{i_3,j_3}
        \]
        So we can change the second index of each point without changing the property "colinear 
        in the generalized sense". As a result, it does not change the property "in good position".
        Then the lemma follows from the definition of $\operatorname{Min}(B)$ in Definition \ref{minimal}.
    \end{proof}
\end{lemma}

\begin{lemma}\label{lemma:judge-good-position}
Let $S$ be a simple surface. If the linear system
\[
\left|2\mu^*H-\mu^*C_{i_1,j_1}-\mu^*C_{i_2,j_2}-\mu^*C_{i_3,j_3}-\mu^*C_{i_4,j_4}\right|
\]
is base point free, the points
\[
p_{i_1,j_1},\ p_{i_2,j_2},\ p_{i_3,j_3},\ p_{i_4,j_4}
\]
are in good position.
\begin{proof}
Set
\[
L:=2\mu^*H-\mu^*C_{i_1,j_1}-\mu^*C_{i_2,j_2}-\mu^*C_{i_3,j_3}-\mu^*C_{i_4,j_4}.
\]

For every choice of three distinct indices $a,b,c\in\{1,2,3,4\}$, if the divisor
\[
\mu^*H-\mu^*C_{i_a,j_a}-\mu^*C_{i_b,j_b}-\mu^*C_{i_c,j_c}
\]
is effective, then every divisor in $|L|$ meets it negatively. This contradicts the assumption that $|L|$ is nef. Therefore no three of the four points are colinear in the generalized sense, and hence they are in good position.
\end{proof}
\end{lemma}

\begin{proposition}\label{S_4}
    Let $S$ be a weak Del Pezzo surface of degree 5. Then the following are equivalent:\\
    (1) There exists a conic bundle structure $g:S\rightarrow \mathbb{P}^1$ of degree 2\\
    (2) The four points $p_{i_1,j_1},p_{i_2,j_2},p_{i_3,j_3},p_{i_4,j_4}$ are in good position. \\
    Moreover, the following holds: if $F$ is a fibre of $g$, such that $\mu(F)$ is a smooth conic in $\mathbb{P}^2$, then $F$ is a smooth fibre.
    \begin{proof}
        $(1) \Rightarrow (2)$: By Proposition~\ref{pro:conic bundles looks like}, the fibration $g$ is induced by the base-point-free linear system
\[
P:=\left|2\mu^*H-\mu^*C_{i_1,j_1}-\mu^*C_{i_2,j_2}-\mu^*C_{i_3,j_3}-\mu^*C_{i_4,j_4}\right|.
\]
Then (2) follows by Lemma \ref{lemma:judge-good-position}.

        $(2) \Rightarrow (1)$: By Lemma \ref{effective}, the divisor $(2\mu^*H-\mu^*C_{i_1,j_1}-\mu^*C_{i_2,j_2}-\mu^*C_{i_3,j_3}-\mu^*C_{i_4,j_4})$ is effective, so we can define a linear system 
        \[P:=|2\mu^*H-\mu^*C_{i_1,j_1}-\mu^*C_{i_2,j_2}-\mu^*C_{i_3,j_3}-\mu^*C_{i_4,j_4}|\]

        \emph{We first prove that $P$ does not have a fixed part.}
        Denote the fixed part of $P$ by $P_F$. Set $P_1:=|\mu^*H-\mu^*C_{i_1,j_1}-\mu^*C_{i_2,j_2}|$ and $P_2:=|\mu^*H-\mu^*C_{i_3,j_3}-\mu^*C_{i_4,j_4}|$. By Lemma \ref{effective}, we know that $P_1,P_2$ are effective. Thus $P_1,P_2$ defines two lines on $\mathbb{P}^2$, i.e. $l_1:=\mu(P_1)$ and $l_2:=\mu(P_2)$.
        Denote the strict transform of $l_1,l_2$ by $\mu$ by $\hat{l}_1,\hat{l}_2$. Since $P\simeq P_1+P_2$, we see that $P_F$ is contained in the union of $\hat{l}_1,\hat{l}_2$ and the exceptional curves. 

        Suppose that $\hat{l}_1\subset P_F$. Since the four points are in good position, $l_1$ passes through at most 2 points in the generalized sense. By definition of $l_1$, 
        \[\mu^*l_1-\mu^*C_{i_1,j_1}-\mu^*C_{i_2,j_2}\]
        is effective. So by Lemma \ref{lemma:passes at least}, $l_1$ passes through at least 2 points. This implies $l_1$ passes through two points exactly, say
$p_{i'_1,j'_1}$ and $p_{i'_2,j'_2}$, with complementary points
$p_{i'_3,j'_3}$ and $p_{i'_4,j'_4}$.
        By Proposition \ref{2nd def of belonging}, we have
        \[
        \hat{l}_1\simeq \mu^*H-\mu^*C_{i'_1,j'_1}-\mu^*C_{i'_2,j'_2}
        \]
        This implies that the movable part is contained in
        \[
        P-\hat{l}_1=\mu^*H-\mu^*C_{i'_3,j'_3}-\mu^*C_{i'_4,j'_4}
        \]
        Take an arbitrary element $M\in |P-\hat{l}_1|$. By Lemma \ref{lemma:passes at least}, the line $\mu(M)$ passes through two points $\operatorname{Min}\{p_{i_3',j_3'},p_{i_4',j_4'}\}$. Thus, the linear system $|P-\hat{l}_1|$ is not movable, 
        We get a contradiction. So $\hat{l}_1 \not\subset P_F$ and for the same reason,  $\hat{l}_2\not\subset P_F$. This implies that $P_F$ is contained in the exceptional curves.

        Then we can take $M$ as a general element of the movable part of $P$ which is irreducible. Define $C:=\mu(M)$, then $C$ is a smooth conic in $\mathbb{P}^2$. Denote the strict transform of $C$ in $S$ by $\hat{C}$. Then we have $M=\hat{C}$ and $P-\hat{C}$ is effective. 

        By Proposition \ref{2nd def of belonging}, we have 
        \begin{align*}
            \begin{split}
                \hat{C}
                \simeq& 2\mu^*H-\sum_{p_{i,j}\in_{gen} C }\mu^*C_{i,j}
            \end{split}
        \end{align*}
        Thus, 
        \begin{align*}
            \begin{split}
                P-\hat{C}\simeq &(2\mu^*H-\mu^*C_{i_1,j_1}-\mu^*C_{i_2,j_2}-\mu^*C_{i_3,j_3}-\mu^*C_{i_4,j_4})\\
                -&(2\mu^*H-\sum_{p_{i,j}\in_{gen} C}\mu^*C_{i,j})\\
                =& -(\sum_{p_{i,j}\not\in_{gen} C}\mu^*C_{i,j})
            \end{split}
        \end{align*}
        So it follows that $C$ passes through four points $p_{i_1,j_1},p_{i_2,j_2},p_{i_3,j_3},p_{i_4,j_4}$ in the generalized sense and $P\simeq \hat{C}$.
        This shows that linear system $P$ does not have a fixed part.

        By Proposition \ref{Pro:!}, we have:
        \[P^2=(2\mu^*H-\mu^*C_{i_1,j_1}-\mu^*C_{i_2,j_2}-\mu^*C_{i_3,j_3}-\mu^*C_{i_4,j_4})^2=0\]
        Since $P$ has no fixed part, we can choose $C_1,C_2$ in $P$ such that they have no common components.The formula $P^2=0$ implies that they are disjoint. Thus, $P$ is base point free. 

        So we have a morphism $f:S\rightarrow \mathbb{P}^k$. 
        Since $P^2=0$, we have $k=1$ and the fibration $f$ is a conic bundle structure of degree 2.
        In addition, for each singular fibre $F_{\alpha}$ of $g$, if the image $\mu(F_{\alpha})$ is a conic $C$, then $F_{\alpha}=\hat{C}$ is irreducible. In particular, $F_{\alpha}$ is smooth, we get a contradiction.
    \end{proof}
\end{proposition}
 
\begin{rk}
    In Proposition \ref{S_4}, since $S$ has degree 5, we only blow up 4 points. Thus, the four points are automatically minimal.
\end{rk}
\begin{proposition}\label{pro: conic 2}
    For a simple surface $S$, there exists a bijection between the following two sets:
    \begin{enumerate}[(1)]
        \item A conic bundle structure $g:S\rightarrow \mathbb{P}^1$ of degree 2.
        \item Four points $p_{i_1,j_1},p_{i_2,j_2},p_{i_3,j_3},p_{i_4,j_4}$ which are in good position and minimal.    
    \end{enumerate}
    
    \begin{proof}
        $(1)\Rightarrow (2)$: By Proposition \ref{pro:conic bundles looks like}, there exists a linear system $P:=|2\mu^*H-\mu^*C_{i_1,j_1}-\mu^*C_{i_2,j_2}-\mu^*C_{i_3,j_3}-\mu^*C_{i_4,j_4}|$ which induces $g$.
        
        By Lemma \ref{lemma:judge-good-position}, the points $p_{i_1,j_1},p_{i_2,j_2},p_{i_3,j_3},p_{i_4,j_4}$ are in good position. So we only have to show they are minimal. If they are not minimal, then by Definition \ref{minimal}, we can assume that $p_{i_1,j_1}=p_{1,k}$, where $2\leq k\leq l(1)$ and there exists an integer $1\leq k'<k$, such that 
        \[
        \{(i_2,j_2),(i_3,j_3),(i_4,j_4)\}\cap \{(1,k'),...,(1,k-1)\}=\emptyset. 
        \]
        So the divisor $\mu^*C_{1,k'}-\mu^*C_{1,k}$
         is effective and we have: 
        \begin{align*}
            \begin{split}
            &P.(\mu^*C_{1,k'}-\mu^*C_{1,k})\\
                =&(2\mu^*H-\mu^*C_{1,k}-\mu^*C_{i_2,j_2}-\mu^*C_{i_3,j_3}-\mu^*C_{i_4,j_4}).(\mu^*C_{1,k'}-\mu^*C_{1,k})=-1
            \end{split}
        \end{align*}
        This shows that $P$ is not nef and we get a contradiction.
        
        $(2)\Rightarrow (1)$: By Theorem \ref{thm:decomposition-mu}, we can contract all the exceptional curves except $E_{i_1,j_1},E_{i_2,j_2},E_{i_3,j_3},E_{i_4,j_4}$ to get a birational morphism $\gamma_1$ from $S$ to a weak Del Pezzo surface of degree 5 which we denote by $S_4$. By Proposition \ref{S_4}, we have a conic bundle structure $g':S_4\rightarrow \mathbb{P}^1$ of degree 2 .
        \begin{center}
    \vspace{0.5em}
    $\xymatrix{
    S \ar[r]^{\gamma_1} \ar[dr]_{g}& S_4 \ar[d]^{g'} \ar[r]^{\gamma_2}& \mathbb{P}^2 \\
    &\mathbb{P}^1  & 
}$
\end{center}
In this way, we get a conic bundle structure $g:S\rightarrow \mathbb{P}^1$ of degree 2.

    \end{proof}  
\end{proposition}

\section{Conic bundle structures on weak Del Pezzo surfaces of degree 4}
Now we concentrate on weak Del Pezzo surfaces of degree $4$. Recall that $\pi:\mathbb{P}(T_X)\rightarrow S$ is the natural projection and $\zeta:= \mathcal{O}_{\mathbb{P}(T_X)}(1)$. 
For a fibration $f:S\rightarrow \mathbb{P}^1$, we denote the relative tangent bundle of $f$ by $T_f$.

\subsection{Conic bundle structures of degree 1}

\begin{setup}\label{setup:p1}
    Let $S$ be a weak Del Pezzo surface of degree $4$, and let $f$ be a conic bundle structure of degree 1 induced by the linear system $|\mu^*H-\mu^*C_{1,1}|$ as in Proposition \ref{conic 1}. For the next statements, we use notation in Notation \ref{notation:F_alpha}. Then for every fibre $F$ of $f$, the image $\mu(F)$ passes through $p_{1,1}$ in the generalized sense. We define
\[
\operatorname{P}(F_{\alpha})
:=\{\, p_{i,j}\mid (i,j)\neq (1,1),\; p_{i,j}\in_{gen} F_{\alpha}\,\},
\]
\end{setup}

\begin{lemma}\label{3-1}
    In the situation of Setup \ref{setup:p1}, we have for each $\alpha\in \Delta$:
    \begin{enumerate}[(1)]
        \item $|\operatorname{P}(F_{\alpha})|\geq 1$;
        \item $M(F_{\alpha})\leq |\operatorname{P}(F_{\alpha})|$, where $M(F_{\alpha})$ is defined in Notation \ref{notation:F_alpha} (2);
        \item $M(F_{\alpha})\leq 2$.
    \end{enumerate}
    \begin{proof}
        (1) Take a singular fibre $F_{\alpha}$ of $f$ and let $l_{\alpha}:=\mu(F_{\alpha})$. Denote $\hat{l_{\alpha}}$ to be the strict transform of $l_{\alpha}$ by $\mu$. Thus, 
        \begin{align}\label{3-1-1}
            F_{\alpha}&= \mu^*l_{\alpha}-\mu^*C_{1,1}\nonumber\\
            &= \hat{l_{\alpha}}+\sum\limits_{p_{i,j}\in_{gen} l_{\alpha} }\mu^*C_{i,j}-\mu^*C_{1,1}\nonumber 
            && \text{by Proposition \ref{2nd def of belonging}}\\
            &= \hat{l}_{\alpha}+\sum_{p_{i,j}\in \operatorname{P}(F_{\alpha})} \mu^*C_{i,j}
        \end{align}
        If $|\operatorname{P}(F_{\alpha})|=0$, then $F_{\alpha}=\hat{l}_{\alpha}$ is a smooth curve, a contradiction. So $|\operatorname{P}(F_{\alpha})|\geq 1$.

        (2) Write $F_{\alpha}=\hat{l}_{\alpha}+\sum\limits_{i,j}a_{i,j}E_{i,j}$. 
        By the definition of $M(F_{\alpha})$, we have $M(F_{\alpha})=\max\{1,a_{i,j}\}$.
        By Proposition \ref{Pro:!} (1), 
        \[\mu^*C_{i,j}=E_{i,j}+...+E_{i,l(i)}\]
        Hence, for each $1\leq i,i' \leq r$, $1\leq j\leq l(i)$ and $1\leq j'\leq l(i')$, the exceptional curve $E_{i,j}$ appears at most one time in the decomposition of $\mu^*C_{i',j'}$. 
        
        Therefore, combined with Formula \eqref{3-1-1}, for any $1\leq i\leq r$ and $1\leq j\leq l(i)$, we have $a_{i,j}\leq |\operatorname{P}(F_{\alpha})|$. Thus, $M(F_{\alpha})\leq |\operatorname{P}(F_{\alpha})|$.

        (3) Every smooth irreducible curve has self intersection number larger than $ -3$ on the weak Del Pezzo surface $S$. So $\mu(F_{\alpha})$ passes through at most three points (containing $p_{1,1}$) in the generalized sense. This shows that $|\operatorname{P}(F_{\alpha})|\leq 2$. By (2), we get the conclusion.
    \end{proof}
\end{lemma}

\begin{thm}\label{positivity 1}
    In the situation of Setup \ref{setup:p1}, there exists an effective divisor $D_1\subset \mathbb{P}(T_S)$ such that $D_1\in |\zeta-\pi^*T_{f_1}|$ and $\zeta|_{D_1}$ is pseudo-effective.
    
    \begin{proof}
        By Proposition \ref{Pro:VMRT}, there exists an effective divisor $D_1\in |\zeta-\pi^*T_{f_1}|$. By Proposition \ref{pro:zeta on D}, to prove $\zeta|_{D_1}$ is pseudo-effective, we only need to show that 
        \[\sum\limits_{\alpha\in \Delta} (1-\frac{1}{2M(F_{\alpha})})\leq 2\]
        By Lemma \ref{3-1}, we have $|\Delta|\leq \sum\limits_{\alpha\in \Delta}M(F_{\alpha})\leq \sum\limits_{\alpha\in \Delta}|\operatorname{P}(F_{\alpha})|=4$. Combined with Lemma \ref{3-1}, we can classify all situations according to the multiplicity of singular fibres:
        \begin{enumerate}
            \item[Case 1] We have exactly two singular fibres of type $M(F_{\alpha})=2$ . Then we have zero singular fibres of type $M(F_{\alpha})=1$ and $\sum\limits_{\alpha\in \Delta} (1-\frac{1}{2M(F_{\alpha})})\leq  \frac{3}{4}+\frac{3}{4}\leq 2$
            
            \item[Case 2] We have one singular fibres of type $M(F_{\alpha})=2$. Then we have at most two singular fibres of type $M(F_{\alpha})=1$ and $\sum\limits_{\alpha\in \Delta} (1-\frac{1}{2M(F_{\alpha})})\leq  \frac{3}{4}+\frac{1}{2}+\frac{1}{2}\leq 2$

            \item[Case 3] We have zero singular fibre of type $M(F_{\alpha})=2$. Then we have at most four singular fibres of type $M(F_{\alpha})=1$ and $\sum\limits_{\alpha\in \Delta} (1-\frac{1}{2M(F_{\alpha})})\leq  \frac{1}{2}+\frac{1}{2}+\frac{1}{2}+\frac{1}{2}=2$
        \end{enumerate}
        Thus, $\zeta|_{D_1}$ is pseudo-effective.
        \end{proof}
\end{thm}

\begin{cor}\label{cor:final 1}
Let $S$ be a weak Del Pezzo surface of degree 4, and $f_1:S\rightarrow \mathbb{P}^1$ is a conic bundle structure of degree 1. If the relative tangent bundle $T_{f_1}$ is effective, then $T_S$ is almost nef.
    \begin{proof}
        Take $k=1$ and $E=T_{f_1}$ in Proposition \ref{all}, we get our conclusion.
    \end{proof} 
\end{cor}

\subsection{Conic bundle structures of degree 2}
 The following key technical lemma uses again Notation \ref{notation:F_alpha} and  Notation \ref{notation:C_i,j}.
\begin{lemma}\label{lemma:deg 5 wdp}
    Let $S$ be a weak Del Pezzo surface of degree 5 and $g:S\rightarrow \mathbb{P}^1$ be the conic bundle structure defined by the linear system $|2\mu^*H-\mu^*C_{i_1,j_1}-\mu^*C_{i_2,j_2}-\mu^*C_{i_3,j_3}-\mu^*C_{i_4,j_4}|$ as in Proposition \ref{pro: conic 2}, then for each singular fibre $F_{\alpha}$ of $g$, the image $\mu(F_{\alpha})$ is composed of two lines $l_1,l_2$ on $\mathbb{P}^2$. Moreover, we have 
    \begin{enumerate}[(1)]
        \item $|\Delta|\leq 3$.
        \item $M(F_{\alpha})\leq 2$.
        \item $M(F_{\alpha})=2$ if and only if $l_1=l_2$.
        \item If there exists a singular fibre $F_{\alpha}$, such that $M(F_{\alpha})=2$, then for any $\alpha'\neq \alpha$, we have $M(F_{\alpha})=1$ and $|\Delta|\leq 2$.
    \end{enumerate}
    \begin{proof}
    (1) For a weak Del Pezzo surface $S$ and a fibration $g:S\rightarrow \mathbb{P}^1$, we always have 
    \[
    \rho(S)\geq\rho(\mathbb{P}^1)+1+\sum_{\alpha\in \Delta}(m_{\alpha}-1)\geq 2+|\Delta|
    \]
    where $m_{\alpha}=\#(\text{irreducible components of }F_{\alpha})$. 
    So $|\Delta|\leq \rho(S)-2=3$.
    \\
    (2)
        The general fibres are just smooth conics, we only need to show this for singular fibres. For a singular fibre $F_{\alpha}$, the image $\mu(F_{\alpha})$ contains two lines $l_1,l_2$ by the last statement in Proposition \ref{S_4}.
        Each line passes through exactly two points in the generalized sense. Up to renumbering, we can assume that $l_1$ passes through $p_{i_1,j_1},p_{i_2,j_2}$ in the generalized sense and $l_2$ passes through $p_{i'_3,j'_3},p_{i'_4,j'_4}$ in the generalized sense. By Proposition \ref{2nd def of belonging} (1), we have 
        \begin{equation}\label{2nd def of belonging-1}
        \begin{cases}
        \mu^*l_1= \hat{l_1}+\mu^*C_{i_1,j_1}+\mu^*C_{i_2,j_2} \\
        \mu^*l_2= \hat{l_2}+\mu^*C_{i_3',j_3'}+\mu^*C_{i_4',j_4'}
        \end{cases}
        \end{equation}
        where $\hat{l_1}, \hat{l_2}$ are the strict transform of $l_1, l_2$ by $\mu$ respectively.
        Then
        \begin{equation}\label{muti 1-1}
                \begin{split}
                    F_{\alpha}\simeq &\mu^*(l_1+l_2)-\mu^*C_{i_1,j_1}-\mu^*C_{i_2,j_2}-\mu^*C_{i_3,j_3}-\mu^*C_{i_4,j_4}\\
                    =& \hat{l_1}+\hat{l_2}+\mu^*C_{i_3',j_3'}+\mu^*C_{i_4',j_4'}-\mu^*C_{i_3,j_3}-\mu^*C_{i_4,j_4}
                \end{split}
        \end{equation}
        By Proposition \ref{Pro:!} (1), it follows that $M(F_{\alpha})\leq 2$.

        (3) When $l_1=l_2$, we have $M(F_{\alpha})=2$. When $l_1\neq l_2$ and $M(F_{\alpha})=2$, by Proposition \ref{Pro:!} (1), we have $i'_3=i'_4$. This forces $i_3=i_4=i_3'=i_4'$, since otherwise $F_{\alpha}$ will not be effective. 

        We also have $j'_3\neq j_3$ and $j'_4\neq j_3$, since otherwise, by Proposition \ref{Pro:!} (1), we have $M(F_{\alpha})=1$. So the points $p_{i'_3,j'_3},p_{i'_4,j'_4},p_{i_3,j_3},p_{i_4,j_4}$ are different. Then we have 
        \[\{p_{i'_3,j'_3},p_{i'_4,j'_4}\}=\{p_{i_1,j_1},p_{i_2,j_2}\}\]
        Thus, the lines $l_1,l_2$ pass both through two points $p_{i_1,j_1},p_{i_2,j_2}$ in the generalized sense. This means that $l_1=l_2$, so we get a contradiction.

        (4) We claim that $m_{\alpha}\geq 3$, 
        where $m_{\alpha}$ is defined in the proof of (1). 
        Since $M(F_{\alpha})=2$, we have $l_1=l_2$ by (3). Thus, in Formula \eqref{2nd def of belonging-1}, we have $(i_1,j_1)=(i_3',j_3')$ and $(i_2,j_2)=(i_4',j_4')$. So Formula \eqref{muti 1-1} becomes 
        \begin{align*}
                \begin{split}
                    F_{\alpha}\simeq &\mu^*(l_1+l_2)-\mu^*C_{i_1,j_1}-\mu^*C_{i_2,j_2}-\mu^*C_{i_3,j_3}-\mu^*C_{i_4,j_4}\\
                    =& 2\hat{l_1}+\mu^*C_{i_1,j_1}+\mu^*C_{i_2,j_2}-\mu^*C_{i_3,j_3}-\mu^*C_{i_4,j_4}
                \end{split}  
        \end{align*}
        Since $F_{\alpha}$ is effective, by Proposition \ref{Pro:!} (1), we can assume that $i_1=i_3$ and $i_2=i_4$. And also by Proposition \ref{Pro:!} (1), $\mu^*C_{i_1,j_1}-\mu^*C_{i_1,j_3}=\mu^*C_{i_2,j_2}-\mu^*C_{i_2,j_4}$ if and only if $(i_1,j_1)=(i_2,j_2)$ and $(i_1,j_3)=(i_2,j_4)$. Thus, 
        \[F_{\alpha}\simeq 2\hat{l_1}+(\mu^*C_{i_1,j_1}-\mu^*C_{i_1,j_3})+(\mu^*C_{i_2,j_2}-\mu^*C_{i_2,j_4})\]
        contains at least three irreducible components. This proves the claim. 
        We also have:
        \[
        \rho(S)=\rho(\mathbb{P}^1)+1+(m_{\alpha}-1)+\sum_{\alpha'\neq \alpha}(m_{\alpha'}-1)\geq 4+|\Delta \backslash \alpha|
        \].
        So $|\Delta \backslash\alpha|\leq \rho(S)-4=1$. Note also that for any $\alpha'\neq \alpha$, we have $M(F_{\alpha}')=1$, otherwise by the same argument, we have $m_{\alpha'}\geq 3$, and therefore 
        \[
        5=\rho(S)-\rho(\mathbb{P}^1)\geq m_\alpha+m_{\alpha'}\geq 6
        \]
        We get a contradiction. 
    \end{proof}
\end{lemma}

\begin{thm}\label{positivity 2}
Let $S$ be a weak Del Pezzo surface of degree $4$ such that four points are in good position and minimal. 
Let $g:S\rightarrow \mathbb{P}^1$ be the fibration defined in Proposition~\ref{pro: conic 2}. 
Then there exists an effective divisor $D_2\in |\zeta-\pi^*T_{g}|$ such that $\zeta|_{D_2}$ is pseudo-effective.
\begin{proof}
By Proposition~\ref{pro:zeta on D} (3), it suffices to show that
\[
\sum_{\alpha\in \Delta}\!\left(1-\frac{1}{2M(F_{\alpha})}\right)\le 2 .
\]

By Proposition~\ref{pro: conic 2}, we may assume that $g$ is induced by four points 
\[
p_{i_1,j_1},p_{i_2,j_2},p_{i_3,j_3},p_{i_4,j_4}
\]
which are in good position and minimal. Let $p_{i_5,j_5}$ be the last point.
As in the proof of Proposition~\ref{pro: conic 2} $(2)\Rightarrow(1)$, we have the diagram
\[
\xymatrix{
S \ar[r]^{\gamma_1} \ar[dr]_{g}& S_4 \ar[d]^{g'} \ar[r]^{\gamma_2}& \mathbb{P}^2 \\
&\mathbb{P}^1  & 
}
\]
where $S_4$ is a weak Del Pezzo surface of degree $5$, and $\gamma_1$ is the contraction of $E_{i_5,j_5}$. 
The morphisms $g'$ and $g$ are conic bundle structures.

Let $F'_{\alpha}\in\Delta'$ and $F_{\alpha}\in\Delta$ be the singular fibres of 
$g'$ and $g$ respectively. 
According to Lemma~\ref{lemma:deg 5 wdp} and the position of $p_{i_5,j_5}$, 
we have the following possibilities.

\begin{enumerate}
\item[Case 1] For every $\alpha\in \Delta'$, we have $M(F'_{\alpha})=1$. 
By Lemma~\ref{lemma:deg 5 wdp} (1), this implies $|\Delta'|\le 3$.

\begin{enumerate}
\item[Case 1.1] The point $p_{i_5,j_5}$ lies on a smooth fibre of $g'$. 
Then $|\Delta|=|\Delta'|+1\le 4$ and $M(F_{\alpha})=1$ for all $\alpha\in \Delta$. Hence
\[
\sum_{\alpha\in \Delta}\!\left(1-\frac{1}{2M(F_{\alpha})}\right)
\le \tfrac12+\tfrac12+\tfrac12+\tfrac12=2 .
\]

\item[Case 1.2] The point $p_{i_5,j_5}$ lies on a singular fibre $F'_{\alpha_0}$ of $g'$. 
Then $|\Delta|=|\Delta'|\le 3$. Since $M(F'_{\alpha_0})=1$, the fibre $F'_{\alpha_0}$ is reduced and has simple normal crossings. 
After blowing up the point $p_{i_5,j_5}$ on $F'_{\alpha_0}$, we obtain
\[
M(F_{\alpha_0})\le 2,
\qquad \text{where }
F_{\alpha_0}=\gamma_1^*F'_{\alpha_0}.
\]
Thus there exists $\alpha_0\in \Delta$ such that $M(F_{\alpha_0})\le 2$ and $M(F_{\alpha})=1$ for $\alpha\ne\alpha_0$. Hence
\[
\sum_{\alpha\in \Delta}\!\left(1-\frac{1}{2M(F_{\alpha})}\right)
\le \tfrac34+\tfrac12+\tfrac12<2 .
\]
\end{enumerate}

\item[Case 2] There exists $\alpha_0\in \Delta'$ such that $M(F'_{\alpha_0})=2$. 
By Lemma~\ref{lemma:deg 5 wdp} (4), we have $|\Delta'|\le 2$.

\begin{enumerate}
\item[Case 2.1] The point $p_{i_5,j_5}$ lies on a smooth fibre of $g'$. 
Then $|\Delta|=|\Delta'|+1\le 3$. Combining with Lemma~\ref{lemma:deg 5 wdp} (4), we obtain that there exists $\alpha_0\in\Delta$ such that $M(F_{\alpha_0})=2$, while $M(F_{\alpha})=1$ for $\alpha\ne\alpha_0$. Hence
\[
\sum_{\alpha\in \Delta}\!\left(1-\frac{1}{2M(F_{\alpha})}\right)
\le \tfrac34+\tfrac12+\tfrac12<2 .
\]

\item[Case 2.2] The point $p_{i_5,j_5}$ lies on a singular fibre of $g'$. 
Then $|\Delta|=|\Delta'|\le 2$. Hence
\[
\sum_{\alpha\in \Delta}\!\left(1-\frac{1}{2M(F_{\alpha})}\right)
\le 1+1=2 .
\]
\end{enumerate}
\end{enumerate}
\end{proof}
\end{thm}

\section{Proof of the Main Theorem}
\begin{thm}\label{thm:2 part final prove}
Let $S$ be a weak Del Pezzo surface of degree 4.
    If there exist four points $p_{i_1,j_1}, p_{i_2,j_2},p_{i_3,j_3},p_{i_4,j_4}$ which are in good position and minimal, then $T_S$ is almost nef.
    \begin{proof}
    Let $f$ be the conic bundle structure defined by $|\mu^*H-\mu^*C_{i_5,1}|$ 
     as in Proposition \ref{conic 1}
     and take $g$ to be the conic bundle structure defined by 
     \[2\mu^*H-\mu^*C_{i_1,j_1}- \mu^*C_{i_2,j_2}-\mu^*C_{i_3,j_3}-\mu^*C_{i_4,j_4}\]
     as in Proposition \ref{pro: conic 2}.
    By Theorem \ref{positivity 1} and \ref{positivity 2}, there exist effective divisors $D_1\in |\zeta-\pi^*T_{f}|$, and $D_2\in |\zeta-\pi^*T_{g}|$ such that $\zeta|_{D_1}$ and $\zeta|_{D_2}$ are pseudo-effective. This implies that $\zeta|_{D_1+D_2}$ is pseudo-effective. 

    Recall that $T_{f}\simeq -K_S+f^*K_{\mathbb{P}^1}+R(f)$ and $T_{g}\simeq -K_S+g^*K_{\mathbb{P}^1}+R(g)$. Denote a general fibre of $f$ by $F_1$ and a general fibre of $g$ by $F_2$. Then by Proposition \ref{Pro:!} (5), we have 
    \begin{align}\label{eq:final-2-1}
        F_2\simeq 2\mu^*H-\mu^*C_{i_1,j_1}- \mu^*C_{i_2,j_2}-\mu^*C_{i_3,j_3}-\mu^*C_{i_4,j_4}= -K_S-(\mu^*H-\mu^*C_{i_5,j_5})
    \end{align}
    So 
        \begin{align*}
T_f + T_g 
&\simeq -K_S + f^*K_{\mathbb{P}^1} - K_S + g^*K_{\mathbb{P}^1} + R(f) + R(g) 
&\qquad& \notag\\
&= -2K_S - 2F_1 - 2F_2 + R(f) + R(g) 
&& K_{\mathbb{P}^1} = \mathcal{O}_{\mathbb{P}^1}(-2) \notag\\
&\ge -2K_S - 2F_1 - 2F_2 
&& R(f),R(g) \ge 0 \notag\\
&\simeq -2K_S - 2(\mu^*H - \mu^*C_{i_5,1}) 
&& \text{by }\eqref{eq:final-2-1}\notag\\
&\quad - 2\bigl(-K_S - (\mu^*H - \mu^*C_{i_5,j_5})\bigr) 
&& \notag\\
&= 2(\mu^*C_{i_5,1} - \mu^*C_{i_5,j_5}) \ge 0. 
\end{align*}
Thus, the divisor $E:=(T_{f}+T_{g})$ is effective and $T_S$ is almost nef by Proposition \ref{all}.
    \end{proof}
\end{thm}

We are left to deal with the case where there are no four points which are in good position, our goal is now to show that there exists a conic bundle structure $f_1:S\rightarrow \mathbb{P}^1$ of degree 1, such that $T_{f_1}$ is effective.

\begin{lemma}\label{same base}
    In the situation of Setup \ref{setup:p1}, let $D$ be a smooth line on $\mathbb{P}^2$. If the two divisors 
    \begin{equation}
        \begin{cases}
            \mu^*D-\mu^*C_{i_1,j_1}-\mu^*C_{i_2,j_2}-\mu^*C_{i_3,j_3}\\
            \mu^*D-\mu^*C_{i_1,j_1}-\mu^*C_{i_2,j_2}-\mu^*C_{i_4,j_4}
        \end{cases}
    \end{equation}
    are effective. Then $i_3=i_4$.
    \begin{proof}
        If $i_3\neq i_4$, then $\operatorname{Supp}(\mu^*C_{i_3,j_3})\cap \operatorname{Supp}(\mu^*C_{i_4,j_4})=\emptyset$. Thus,
        \[
        (\mu^*D-\mu^*C_{i_1,j_1}-\mu^*C_{i_2,j_2})-\mu^*C_{i_3,j_3}-\mu^*C_{i_4,j_4}
        \]
        is effective. Take $B=\{p_{i_1,j_1},p_{i_2,j_2},p_{i_3,j_3},p_{i_4,j_4}\}$ and use Lemma \ref{lemma:passes at least}, the line $D$ passes through at least 4 points in the generalized sense. This yields an irreducible curve whose self intersection number is less than $(-3)$. We get a contradiction.
    \end{proof}
\end{lemma}
Recall that for a simple surface, we defined an integer $r$ in Proposition \ref{graph}.
\begin{proposition}\label{r > 5}
    Let $S$ be a simple surface. If $r\geq 5$, then there exists a conic bundle structure of degree 2 $f:S\rightarrow \mathbb{P}^1$.
    \begin{proof}
    We claim that for the five points $p_{1,1},p_{2,1},p_{3,1}, p_{4,1},p_{5,1}$, there are four points of them which are in good position. We argue by contradiction and assume that none of four points of them are in good position.

        Take $p_{1,1},p_{2,1},p_{3,1}$ to be three points which are not colinear and let $l_{s,t}$ be the lines defined by $p_{s,1},p_{t,1}$ on $\mathbb{P}^2$, for $1\leq s<t\leq 3$. 
        Then the points $p_{4,1}$, $p_{5,1}$ belong to the union of $l_{12},l_{12},l_{23}$. Otherwise there will be four points in good position. Up to renumbering, we can assume that $p_{4,1}\in l_{12}$ and $p_{5,1}\in l_{23}$. We draw a picture as follows:

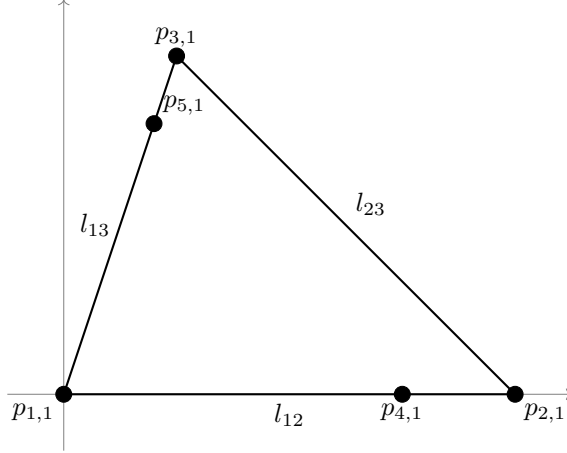
\begin{tikzpicture}[scale=1.5]

\draw[->, very thin, gray] (-0.5,0) -- (4.5,0);
\draw[->, very thin, gray] (0,-0.5) -- (0,3.5);

\coordinate (p1) at (0,0);
\coordinate (p2) at (4,0);
\coordinate (p3) at (1,3);

\coordinate (p4) at (3,0); 

\coordinate (p5) at (0.8,2.4); 

\draw[thick] (p1) -- (p2) node[midway, below] {$l_{12}$}; 
\draw[thick] (p2) -- (p3) node[midway, above right] {$l_{23}$}; 
\draw[thick] (p1) -- (p3) node[midway, left] {$l_{13}$}; 

\filldraw[black] (p1) circle (2pt) node[below left] {$p_{1,1}$};
\filldraw[black] (p2) circle (2pt) node[below right] {$p_{2,1}$};
\filldraw[black] (p3) circle (2pt) node[above] {$p_{3,1}$};
\filldraw[black] (p4) circle (2pt) node[below] {$p_{4,1}$};
\filldraw[black] (p5) circle (2pt) node[above right] {$p_{5,1}$};

\end{tikzpicture}

Then we can easily check that $p_{2,1},p_{3,1},p_{4,1},p_{5,1}$ are in good position. We get a contradiction.
    \end{proof}
\end{proposition}
Recall that a weak Del Pezzo surface $S$ of degree 4 is the blow up of 5 points $p_{i_1,j_1},...,p_{i_5,j_5}$. Let $p_{i_1,j_1}=p_{1,1}$. For a conic bundle structure $f:S\rightarrow \mathbb{P}^1$ of degree 1, we can assume that it is induced by the linear system $|\mu^*H-\mu^*C_{1,1}|$, and by Notation \ref{notation:relative tangent bundle}, we have 
\begin{align}\label{eq: formula T_f}
    \begin{split}
        T_{f_1}&=-K_S-2(\mu^*H-\mu^*C_{1,1})+R(f_1)\\
        &=\mu^*H+\mu^*C_{1,1}+R(f_1)-(\mu^*C_{i_2,j_2}+\mu^*C_{i_3,j_3}+\mu^*C_{i_4,j_4}+\mu^*C_{i_5,j_5})
    \end{split}
\end{align}

\begin{lemma}\label{lemma:next-1}
Let $S$ be a simple surface. 
Assume that the points 
$p_{i_1,j_1}, p_{i_2,j_2}, p_{i_3,j_3}$ 
are colinear in the generalized sense. 
Suppose further that the four points 
$p_{i_2,j_2}, p_{i_3,j_3}, p_{i_4,j_4}, p_{i_5,j_5}$ 
are not in good position. 
Then one of the following holds:
\begin{enumerate}[(1)]
\item One of the divisors
\begin{align*}
    \begin{cases}
        \mu^*H-\mu^*C_{i_2,j_2}-\mu^*C_{i_4,j_4}-\mu^*C_{i_5,j_5}\\
        \mu^*H-\mu^*C_{i_3,j_3}-\mu^*C_{i_4,j_4}-\mu^*C_{i_5,j_5}
    \end{cases}
\end{align*}
is effective; or
\item $\{i_4,i_5\}\cap \{i_1\} \neq \emptyset$.
\end{enumerate}
\begin{proof}
    Since four points $p_{i_2,j_2}, p_{i_3,j_3}, p_{i_4,j_4}, p_{i_5,j_5}$ are not in good position, by Definition \ref{def:good position}, one of the divisors 
    \begin{align*}
        \begin{cases}
        \mu^*H-\mu^*C_{i_2,j_2}-\mu^*C_{i_4,j_4}-\mu^*C_{i_5,j_5}\\
        \mu^*H-\mu^*C_{i_3,j_3}-\mu^*C_{i_4,j_4}-\mu^*C_{i_5,j_5}
    \end{cases}
    \end{align*}
    is effective, which implies (1), or one of the divisiors 
    \begin{align*}
        \begin{cases}
        \mu^*H-\mu^*C_{i_2,j_2}-\mu^*C_{i_3,j_3}-\mu^*C_{i_4,j_4}\\
        \mu^*H-\mu^*C_{i_2,j_2}-\mu^*C_{i_3,j_3}-\mu^*C_{i_5,j_5}
    \end{cases}
    \end{align*}
    is effective. Without loss of generality, we assume that $\mu^*H-\mu^*C_{i_2,j_2}-\mu^*C_{i_3,j_3}-\mu^*C_{i_4,j_4}$ is effective. Take $l:=\mu(\mu^*H-\mu^*C_{i_2,j_2}-\mu^*C_{i_3,j_3}-\mu^*C_{i_4,j_4})$. 
    
    By assumption the points 
$p_{i_1,j_1}, p_{i_2,j_2}, p_{i_3,j_3}$ 
are collinear in the generalized sense, so the divisor 
\[\mu^*H-\mu^*C_{i_1,j_1}-\mu^*C_{i_2,j_2}-\mu^*C_{i_3,j_3}\]
is effective. Take $l':=\mu(\mu^*H-\mu^*C_{i_1,j_1}-\mu^*C_{i_2,j_2}-\mu^*C_{i_3,j_3})$.

By Lemma \ref{lemma:passes at least}, the line $l$ and $l'$ pass through the points $\operatorname{Min}(\{p_{i_1,j_1},p_{i_2,j_2}\})$(See Definition \ref{minimal} for notation $\operatorname{Min}$). So $l=l'$ and the divisors 
\begin{align*}
    \begin{cases}
        \mu^*l-\mu^*C_{i_1,j_1}-\mu^*C_{i_2,j_2}-\mu^*C_{i_3,j_3}\\
        \mu^*l-\mu^*C_{i_2,j_2}-\mu^*C_{i_3,j_3}-\mu^*C_{i_4,j_4}
    \end{cases}
\end{align*}
are effective.
By Lemma \ref{same base}, we get either $i_4=i_1$ or $i_5=i_1$, which is (2).
\end{proof}
\end{lemma}

\begin{thm}\label{thm:bad cases}
    Let $S$ be a weak Del Pezzo surface of degree 4, and there are no four points which are in good position. Then there exists a conic bundle structure 
    \[
    f_1:S\rightarrow \mathbb{P}^1
    \]
    such that $T_{f_1}$ is effective.
    \begin{proof}
        Since there are no four points which are in good position, we can assume that there exists a line $l$ which passes through three points $p_{i_1,j_1},p_{i_2,j_2},p_{i_3,j_3}$ and does not pass through the other points. By Lemma \ref{lemma:passes at least}, the three points $p_{i_1,j_1},p_{i_2,j_2},p_{i_3,j_3}$ can be chosen to be minimal. so there are 3 possibilities for the geometric structure of the line $l$. 
        \begin{enumerate}
            \item [Type 1] \begin{tikzpicture}[baseline=(current bounding box.center)]
  \node (A) at (0,0) {$\bullet$};
  \node (B) at (1,0) {$\bullet$};
  \node (C) at (2,0) {$\bullet$};  
  \node[below] at (A.south) {$p_{1,1}$};
  \node[below] at (B.south) {$p_{1,2}$};
  \node[below] at (C.south) {$p_{1,3}$};

  \draw[->] (B)--(A);
  \draw[->] (C)--(B);
\end{tikzpicture}\\
\item [Type 2] \begin{tikzpicture}[baseline=(current bounding box.center)]
  \node (A) at (0,0) {$\bullet$};
  \node (B) at (1,0) {$\bullet$};
  \node (C) at (2,0) {$\bullet$};  
  \node[below] at (A.south) {$p_{1,1}$};
  \node[below] at (B.south) {$p_{2,1}$};
  \node[below] at (C.south) {$p_{2,2}$};

  \draw[->] (C)--(B);
\end{tikzpicture}\\
\item [Type 3]\begin{tikzpicture}[baseline=(current bounding box.center)]
  \node (A) at (0,0) {$\bullet$};
  \node (B) at (1,0) {$\bullet$};
  \node (C) at (2,0) {$\bullet$};  
  \node[below] at (A.south) {$p_{1,1}$};
  \node[below] at (B.south) {$p_{2,1}$};
  \node[below] at (C.south) {$p_{3,1}$};
\end{tikzpicture}\\
        \end{enumerate}
We make a case distinction:

\textbf{Case 1: There exists a line of Type 1.}
Take $f_1$ to be the conic bundle structure induced by $|\mu^*H-\mu^*C_{1,1}|$. 
    Then there exists a singular fibre $F_{\alpha_0}$ induced by $l$, such that 
    \begin{equation}
\begin{alignedat}{2}
F_{\alpha_0}
&= \mu^*l - \mu^*C_{1,1} \\
&= \hat{l} + \mu^*C_{1,2} + \mu^*C_{1,3}
&\quad &\text{since } 
\mu^*l = \hat{l} + \mu^*C_{1,1} + \mu^*C_{1,2} + \mu^*C_{1,3} \\
&= \hat{l} + E_{1,2} + 2\mu^*C_{1,3}
&\quad &\text{by Lemma~\ref{Pro:!} (1)}.
\end{alignedat}
\end{equation}
where $\hat{l}$ is the strict transform of $l$ by $\mu$. Thus, 
\[
R(f_1)=\sum_{\alpha\in \Delta}(F_{\alpha}-(F_{\alpha})_{red})\geq F_{\alpha_0}-(F_{\alpha_0})_{red}= \mu^*C_{1,3}
\]
Recall that we defined the integer $l(1)$ in the Notation \ref{notation:C_i,j} and we do a classification with $l(1)$.
    \begin{enumerate}
        \item [Case 1.1] When $l(1)=3$, we have $\{i_4,i_5\}\cap \{1\}=\emptyset$. The three points $p_{1,1},p_{1,2},p_{1,3}$ are colinear in the generalized sense and the four points $p_{1,1},p_{1,2},p_{i_4,j_4},p_{i_5,j_5}$ are not in good position. So by Lemma \ref{lemma:next-1} and Lemma \ref{lemma:passes at least}, we always have 
        \begin{align}\label{eq:bad cases-1.1}
            \mu^*H-\mu^*C_{1,1}-\mu^*C_{i_4,j_4}-\mu^*C_{i_5,j_5}\geq 0
        \end{align}
        So 
        \begin{align*}
T_{f_1}\simeq\;&(\mu^*H-\mu^*C_{1,1}-\mu^*C_{i_4,j_4}-\mu^*C_{i_5,j_5})\\
&+(\mu^*C_{1,1}-\mu^*C_{1,2})
+(\mu^*C_{1,1}-\mu^*C_{1,3})
+R(f_1)
&&\text{by \eqref{eq: formula T_f}}
\end{align*}
Thus, $T_{f_1}$ is effective by Lemma~\ref{Pro:!} (1) and \eqref{eq:bad cases-1.1}.
        \item [Case 1.2] When $l(1)\geq 4$, we take $p_{i_4,j_4}=p_{1,4}$. By Lemma \ref{effective} and $R(f_1)\geq \mu^*C_{1,3}$, 
        \begin{align*}
                T_{f_1}&\simeq (\mu^*H-\mu^*C_{1,1}-\mu^*C_{i_5,j_5})+(\mu^*C_{1,1}-\mu^*C_{1,2})\\
                &+(\mu^*C_{1,1}-\mu^*C_{1,4})+(R(f_1)-\mu^*C_{1,3})
                &&\text{by \eqref{eq: formula T_f}}
        \end{align*}
    Thus, $T_{f_1}$ is effective by {by Lemma~\ref{Pro:!} (1) and Lemma \ref{effective}.}
    \end{enumerate}
    \textbf{Case 2: There exists a line of Type 2.}
    Take $f_1$ to be the conic bundle structure induced by $|\mu^*H-\mu^*C_{1,1}|$.
    We use the same technique in Case 1 to get  $R(f_1)\geq \mu^*C_{2,2}$. We make a distinction depending on $l(1)$.
    \begin{enumerate}
        \item [Case 2.1] When $l(1)=1$, we have $\{i_4,i_5\}\cap \{1\}=\emptyset$. The three points $p_{1,1},p_{2,1},p_{2,2}$ are colinear in the generalized sense and four points $p_{2,1},p_{2,2},p_{i_4,j_4},p_{i_5,j_5}$ are not in good position. Thus, by Lemma \ref{lemma:next-1} and Lemma \ref{lemma:passes at least}, we always have 
        \begin{align}\label{eq:bad cases-2.1}
            \mu^*H-\mu^*C_{2,1}-\mu^*C_{i_4,j_4}-\mu^*C_{i_5,j_5}\geq 0
        \end{align}
 So,
        \begin{align*}
                T_{f_1}&\simeq (\mu^*H-\mu^*C_{2,1}-\mu^*C_{i_4,j_4}-\mu^*C_{i_5,j_5})\\
                &+\mu^*C_{1,1}+(R(f_1)-\mu^*C_{2,2}) &&\text{by \eqref{eq: formula T_f} }
        \end{align*}
        Thus $T_{f_1}$ is effective by $R(f_1)\geq \mu^*C_{2,2}$ and \eqref{eq:bad cases-2.1}.
        \item [Case 2.2] $l(1)=2$, we can assume that $p_{i_4,j_4}=p_{1,2}$. Then by Formula \eqref{eq: formula T_f} and Proposition \ref{Pro:!} (1),
        \begin{align*}
                T_{f_1}&\simeq (\mu^*H-\mu^*C_{2,1}-\mu^*C_{i_5,j_5})\\
                &+(\mu^*C_{1,1}-\mu^*C_{1,2})+(R(f_1)-\mu^*C_{2,2}) &&\text{by \eqref{eq: formula T_f}}
        \end{align*}
        Thus, $T_{f_1}$ is effective by Lemma~\ref{Pro:!} (1), Lemma \ref{effective} and $R(f_1)\geq \mu^*C_{2,2}$.
        \item [Case 2.3] $l(1)=3$, then we have $p_{i_4,j_4}=p_{1,2}$ and $p_{i_5,j_5}=p_{1,3}$. Then the divisor 
        \[
        \mu^*H-\mu^*C_{2,1}-\mu^*C_{1,2}-\mu^*C_{1,3}
        \]
        is not effective. Otherwise, by Lemma \ref{lemma:passes at least}, the effective divisor defines a line which passes through points $\operatorname{Min}(p_{2,1},p_{1,2},p_{1,3})=\{p_{2,1},p_{1,1},p_{1,2}\}$. Thus,
        \[
        \mu^*H-\mu^*C_{2,1}-\mu^*C_{1,1}-\mu^*C_{1,2}
        \]
        is effective. By assumption, we have 
        \[
        \mu^*H-\mu^*C_{2,1}-\mu^*C_{1,1}-\mu^*C_{2,2}
        \]
        is effective. Thus, by Lemma \ref{same base}, we get a contradiction. 

        Now recall that the three points $p_{1,1},p_{2,1},p_{2,2}$ are colinear in the generalized sense and four points $p_{2,1},p_{2,2},p_{1,2},p_{1,3}$ are not in good position. Thus, by Lemma \ref{lemma:next-1}, the divisor
        \[\mu^*H-\mu^*C_{1,1}-\mu^*C_{1,2}-\mu^*C_{1,3}\]
        is effective. 
        This defines a line $l'\subset \mathbb{P}^2$, such that $l'$ passes through points $p_{1,1},p_{1,2},p_{1,3}$ in the generalized sense. As a result, $l'$ is a line of type 1. The present case then follows from the argument in Case 1.
    \end{enumerate}
    \textbf{Case 3: There exists a line of Type 3.}
    Take $f_1$ to be the conic bundle structure induced by $|\mu^*H-\mu^*C_{1,1}|$.
    Since there are only 5 points, we can assume that $l(1)=1$. Thus, we have $\{i_4,i_5\}\cap \{1\}=\emptyset$. The three points $p_{1,1},p_{2,1},p_{3,1}$ are colinear in the generalized sense and the four points $p_{2,1},p_{3,1},p_{i_4,j_4},p_{i_5,j_5}$ are not in good position. By Lemma \ref{lemma:next-1}, we have one of the following divisors 
    \begin{align*}
        \begin{cases}
            \mu^*H-\mu^*C_{2,1}-\mu^*C_{i_4,j_4}-\mu^*C_{i_5,j_5}\\
            \mu^*H-\mu^*C_{3,1}-\mu^*C_{i_4,j_4}-\mu^*C_{i_5,j_5}
        \end{cases}
    \end{align*}
    is effective. Without loss of generality, we assume that 
    \[\mu^*H-\mu^*C_{2,1}-\mu^*C_{i_4,j_4}-\mu^*C_{i_5,j_5}\]
    is effective. So this gives another line $l'$. If the line $l'$ is Type 3, then we have $r=5$. By Proposition \ref{r > 5}, there are four points in good position, so we get a contradiction. Thus, $l'$ is Type 1 or 2 and we conclude by Case 1 or 2.
    \end{proof}
\end{thm}

\begin{thm}
    For every weak Del Pezzo surface $S$ of degree $4$, the tangent bundle $T_S$ is almost nef.
    \begin{proof}
        If there are four points in good position, then by Theorem \ref{thm:2 part final prove}, we have $T_S$ is almost nef. If for any four points, they are not in good position, then by Theorem \ref{thm:bad cases}, there exists a conic bundle structure $f:S\rightarrow \mathbb{P}^1$ such that $T_f$ is effective. Thus, by Corollary \ref{cor:final 1}, we have $T_S$ is almost nef.
    \end{proof}
\end{thm}

\bibliographystyle{alpha}
\bibliography{a}

\bigskip

\noindent
Qimin Zhang\\
Université Côte d’Azur\\
CNRS, Laboratoire J.-A. Dieudonné (LJAD)\\
06108 Nice, France\\
\textit{E-mail address}: \texttt{Qimin.ZHANG@univ-cotedazur.fr}

\end{document}